\documentclass[leqno]{amsart}
\usepackage{amsfonts,amssymb,amsmath,amsgen,amsthm}
\usepackage{hyperref,color}
\usepackage{pdfsync}
\usepackage{bbm}
\usepackage{bm}

\usepackage[normalem]{ulem}

\theoremstyle{plain}
\newtheorem{theorem}{Theorem}[section]
\newtheorem{definition}[theorem]{Definition}
\newtheorem{lemma}[theorem]{Lemma}
\newtheorem{corollary}[theorem]{Corollary}
\newtheorem{proposition}[theorem]{Proposition}
\newtheorem{hyp}[theorem]{Assumption}
\theoremstyle{remark}
\newtheorem{remark}[theorem]{Remark}

\def\C{{\mathbb C}}
\def\R{{\mathbb R}}
\def\N{{\mathbb N}}
\def\Z{{\mathbb Z}}

\def\({\left(}
\def\){\right)}
\def\<{\left\langle}
\def\>{\right\rangle}

\def\1{{\mathbf 1}}

\def\d{{\partial}}
\def\eps{\varepsilon}

\DeclareMathOperator{\RE}{Re}
\DeclareMathOperator{\IM}{Im}
\DeclareMathOperator{\diver}{div}

\numberwithin{equation}{section}

\date\today

\title[Existence for QMHD]{Existence and Stability of almost finite energy weak solutions to the Quantum Euler-Maxwell system}

\author[P.~Antonelli]{Paolo Antonelli}
\address[P.~Antonelli]{Gran Sasso Science Institute \\ via Crispi 7 \\ 67100 L'Aquila (Italy).}
\email{paolo.antonelli@gssi.it}
\author[P.~Marcati]{Pierangelo Marcati}
\address[P.~Marcati]{Gran Sasso Science Institute \\ via Crispi 7 \\ 67100 L'Aquila (Italy).}
\email{pierangelo.marcati@gssi.it}
\author[R.~Scandone]{Raffaele Scandone}
\address[R.~Scandone]{Gran Sasso Science Institute \\ via Crispi 7 \\ 67100 L'Aquila (Italy).}
\email{raffaele.scandone@gssi.it}

\begin{document}


\begin{abstract}
We prove the existence of global in time, finite energy, weak solutions to a quantum magnetohydrodynamic system (QMHD) with large data, modeling a charged quantum fluid interacting with a self-generated electromagnetic field. The analysis of QMHD relies upon the use of Madelung transformations. The rigorous derivation requires non-trivial smoothing estimates, which are obtained by assuming slightly higher regularity for the electromagnetic potential.
These assumptions are motivated by the nonlinear dependence of the hydrodynamic system in terms of the underlying wave function dynamics, which is supercritical with respect to the bare energy bounds. 
\newline
Due to quantum effects on the dispersive properties of QMHD, our approach requires neither smallness nor high regularity, unlike a large amount of existing literature for Euler-Maxwell's classical system. In fact, the difficulty posed by the presence of the nonlinear electromagnetic force field (Lorentz) severely restricts the possibility to get existence and stability results in the general framework of finite energy solutions. In the classical case the dispersion is not able to deal with the transport of a non-trivial vorticity, therefore almost GWP holds in a life span, reciprocal of the amplitude of the vorticity. GWP can be proved in the irrotational case, where in any case  smallness and high regularity assumptions are needed. 
\newline
For quantum MHD system the irrotationality and the presence of a highly nonlinear quantum stress tensor induce much stronger dispersive properties, as a byproduct of a close relationship with the classical Maxwell-Schr\"odinger system. 
Therefore the core argument is shifted to the analysis of the nonlinearities related to the formulation of the hydrodynamic variables through the Madelung transformations. The analysis carried out in section 4 shows that it is necessary to go through non-trivial smoothing estimates and these require us to assume regularity conditions, just above the energy norms, for the initial data of the Maxwellian electromagnetic potential.
In the same regime of regularity, with the help of suitable local smoothing estimates, we also prove stability of both the hydrodynamic variables and the Lorentz force associated with the electromagnetic field.
\end{abstract}

\subjclass[2020]{35Q40, 35Q35, 76Y05, 82D10}

\keywords{Quantum hydrodynamics, Quantum MHD, Maxwell-Schr\"odinger, Madelung transformation, Finite-energy weak solutions, Large data}

\maketitle
\section{Introduction and main result}
\subsection{Problem setup}
In this paper we consider the Cauchy problem associated to the following quantum magnetohydrodynamic (QMHD) system in three space dimensions
\begin{equation}\label{eq:qmhd}
\begin{cases}
\partial_t\rho+\diver J=0\\
\partial_t J+\diver\Big(\frac{J\otimes J}{\rho}\Big)+\nabla P(\rho)=\rho E+J\wedge B+\frac{1}{2}\rho\nabla\Big(\frac{\Delta\sqrt{\rho}}{\sqrt{\rho}}\Big)\\
\diver E=\rho,\quad\nabla\wedge E=-\partial_t B\\
\diver B=0,\quad\nabla\wedge B=J+\partial_t E,
\end{cases}
\end{equation}
with prescribed initial data
\begin{equation}\label{eq:qmhd_id}
\rho_{|t=0}=\rho_0, \;J_{|t=0}=J_0,\; E_{|t=0}=E_0, \;B_{|t=0}=B_0.
\end{equation}
The unknowns $\rho, J$ denote the charge and current densitites, whereas $E, B$ are the electric and magnetic fields. 
The pressure term $P=P(\rho)$ is assumed to be barotropic and for simplicity it satisfies a $\gamma-$law, i.e.~ $P(\rho)=\frac{\gamma-1}{\gamma}\rho^{\gamma}$, with $\gamma\in(1, 3)$. 
The equation for the current density presents two contributions on its right hand side, the former one given by the Lorentz force
\begin{equation}\label{eq:lor}
F_L:=\rho E+J\wedge B,
\end{equation}
and the latter being the quantum pressure term, given by the third order nonlinear operator in $\rho$.
The QMHD system \eqref{eq:qmhd} bears no dissipative or relaxation effects and indeed its total energy, given by
\begin{equation}\label{eq:en}
\mathcal E(t)=\int_{\R^3}\frac12|\nabla\sqrt{\rho}|^2+\frac12\frac{|J|^2}{\rho}+f(\rho)+\frac12|E|^2+\frac12|B|^2\,dx,
\end{equation}
is formally conserved along the flow of solutions to \eqref{eq:qmhd}. Here the internal energy $f(\rho)$ is determined by the pressure, according to the formula $f(\rho)=\rho\int_0^\rho P(s)/s^2\,ds$.
\newline
The QMHD system \eqref{eq:qmhd} describes a (positively) charged quantum fluid interacting with its self-generated electromagnetic field. It is a prototypical model for a quantum plasma, arising for instance in the description of dense astrophysical objects such as white dwarf stars \cite{SE3}. 
In particular, the introduction of the quantum term is motivated by the fact that in such contexts the thermal de Broglie wavelength becomes comparable (larger than or equal) to the typical interatomic distance \cite{NP}. 
Unlike classical fluid dynamics, in this framework the term $P(\rho)$ describes the quantum statistical pressure, which is not of thermal origin (hence does not depend on the temperature) but comes from the electron degeneracy, due to Heisenberg's uncertainty principle and Pauli's exclusion principle \cite{SE3}. For instance, for non-relativistic degenerate electron gases, the quantum statistical pressure in the zero temperature limit is given by $P(\rho)\sim\rho^{(d+2)/d}$ \cite{CHM, MH}, where $d$ denotes the space dimension.
A hydrodynamic model of this type has been used by Feynman \cite{Fey1,Fey2} in the study of superconductivity. In fact, in superconductors the electrons in the ground state come in pairs with opposite spin. Since the resulting spin is zero, the ``particle" obeys Bose statistics. Just as Maxwell's equations describe the motion of many photons in the same state, the QMHD system are macroscopic equations for many superconducting ``particles" in the same state.

\subsection{Comparison with similar compressible fluid models}
If the quantum term is neglected in the equation for the current density, then system \eqref{eq:qmhd} reduces to the well known Euler-Maxwell system for ions \cite[Chap. 9]{Bitt}, that shares several mathematical difficulties with its quantum counterpart.
For instance, due to the absence of dissipative terms in both systems, the a priori bounds derived from the physical entropies are in general not sufficient to prove the existence of global in time weak solutions (see also \cite{UWK} where a damping term allows the authors to study the system by the energy method). 
In order to obtain suitable estimates, we thus need to exploit other physical effects encoded in the system, such as dispersion.
This idea has been carefully exploited by Germain and Masmoudi in \cite{GM} for the classical Euler-Maxwell model. In \cite{GM} the global well-posedness was proved by combining higher order energy bounds and dispersive estimates (yielded by an underlying system of Klein-Gordon equations). 
A similar strategy was previously used to study the Euler-Poisson system by Guo \cite{G} and Guo and Pausader \cite{GP}, and was further implemented in more complex models, such as the two-fluid Euler-Maxwell system \cite{GIP}.
Let us mention that all such results are concerned with irrotational solutions. 
This restriction allows to eliminate the normal mode related to the vorticity of the fluid, that is only transported and hence it does not decay.
For non-irrotational solutions (i.e., solutions with vorticity), then it is not possible to obtain a global well-posedness result, see for instance \cite{IL} where their lifespan is inversely proportional to the size of the initial vorticity. Similar result is analysed in the two dimensional case in \cite{Zheng}.
\newline
However, all those previous results \cite{G, GP, GM, GIP} require a smallness assumption on the initial data, that is actually necessary since large perturbations of the equilibrium solution may produce a breakdown of regularity at finite time \cite{GTZ}, in the same spirit of the seminal paper by Sideris \cite{Sid}.
\newline 
Let us notice that the presence of the quantum term in \eqref{eq:qmhd} modifies the dispersion relation for the irrotational part of the flow, that is now given by
\begin{equation*}
\omega(\xi)=\left(c_s^2|\xi|^2+1+\frac14|\xi|^4\right)^{1/2},
\end{equation*}
where $c_s=\sqrt{P'(1)}$ is the speed of sound waves.
The reader should compare this formula with the Klein-Gordon type dispersion relation satisfied by system (1.3) in \cite{GM}, where the quartic part is missing.
The quantum term thus yields better dispersive properties as compared to the classical Euler-Maxwell system, see for instance \cite{AHM, AHM_Brix, H_2D} where such a property was exploited to study acoustic oscillations for quantum fluid systems, in the low Mach number regime. 
We remark that a similar dispersion relation also describes the excitations in a Bose condensed gas, as derived in the seminal paper by Bogoliubov \cite{Bog}, see also \cite{BBCS}.
\newline
However, this augmented dispersion is still not sufficient to yield satisfactory a priori bounds and prove the existence of global solutions for arbitrary initial data. Following a strategy as in \cite{GM} and other aforementioned papers, would not allow us to go beyond an analogue global existence result, which considers initial data that are small regular perturbations of stationary solutions.
\subsection{Quantum fluid models and the Madelung transform}
Our aim is to prove the existence of global in time, finite energy weak solutions to \eqref{eq:qmhd}, without restrictions on the size or any higher assumptions on the initial data. In particular we shall not require the charge density to be a (small and regular) perturbation of a constant.
\newline
Our approach exploits the correspondence between the QMHD system \eqref{eq:qmhd} and a wave function dynamics provided by the nonlinear Maxwell-Schr\"odinger system (see \eqref{eq:clas_MS} below), by means of the so called Madelung transformations \cite{Mad}. 
This has been a well-established and successful strategy, see for instance \cite{AM_CMP, AM_2D} where a class of quantum fluid models was studied by exploiting their analogy with an underlying nonlinear Schr\"odinger-Poisson system. 
More precisely, given a finite energy solution $\psi$ to the wave function dynamics, then its associated momenta, defined by $\rho=|\psi|^2$, $J=\IM(\bar\psi\nabla\psi)$ solve the quantum hydrodynamics (QHD) system in the weak sense.
This correspondence, although formally obtained by Madelung in \cite{Mad}, is rigorously achieved by using the polar factorization method \cite{AM_CMP, AM_2D}, that overcomes the mathematical difficulty of defining the velocity field in the vacuum region -- see Section \ref{sse:polar} for more details. 
\newline
For the QMHD system \eqref{eq:qmhd},  the underlying wave function dynamics is determined by the following nonlinear Maxwell-Schr\"odinger system
\begin{equation}\label{eq:clas_MS}
	\begin{cases}
		i\partial_t \psi = -\frac{1}{2}\Delta_A\psi+\phi \psi+|\psi|^{2(\gamma-1)}\psi\\
		-\Delta\phi-\partial_t\diver A=\rho\\
		\square A+\nabla(\partial_t\phi+\diver A)=J,
	\end{cases}
\end{equation}
where the source terms appearing on the right hand side of the Maxwell's equations are again determined by means of the Madelung transformation,
\begin{equation}\label{eq:mad}
	\rho=|\psi|^2, \quad
	J=\RE(\bar\psi(-i\nabla-A)\psi)
\end{equation}
and $-\Delta_A:=(-i\nabla-A)^2$ is the magnetic (positive definite) Laplacian. 
\newline
Let us notice how the Madelung transformations in \eqref{eq:mad} are modified according to the presence of an electromagnetic potential. 
The electromagnetic fields $(E, B)$ are determined from the potentials $(\phi, A)$ through the usual formulas
\begin{equation}\label{eq:em_f}
E=-\nabla\phi-\d_tA, \quad B=\nabla\wedge A.
\end{equation}
System \eqref{eq:clas_MS} is invariant under the following gauge transformation
\begin{equation}\label{eq:gauge}
	(\psi,\phi,A)\mapsto (e^{i\lambda}\psi,\phi-\partial_t\lambda,A+\nabla\lambda).
\end{equation}
It is straightforward to see that the same transformation leaves invariant also the definition of the hydrodynamic quantities $(\rho, J, E, B)$, so that the choice of the gauge for \eqref{eq:clas_MS} does not affect the study of \eqref{eq:qmhd}. For our convenience here we consider the Coulomb gauge, i.e.~$\diver A=0$, under which the nonlinear Maxwell-Schr\"odinger system becomes
\begin{equation}\label{eq:MS}
	\begin{cases}
		i\partial_t \psi = -\frac{1}{2}\Delta_A\psi+\phi \psi+|\psi|^{2(\gamma-1)}\psi\\
		\square A=\mathbb{P}J\\
		\diver A=0,
	\end{cases}\quad t\in\R,x\in\R^3,
\end{equation}
where now $\phi:=(-\Delta)^{-1}\rho$  and $\mathbb{P}:=\mathbb{I}-\nabla\diver\Delta^{-1}$ denotes the Helmholtz projection onto solenoidal vector fields. 
\newline
Unfortunately, the approach in \cite{AM_CMP, AM_2D} cannot be applied here in a straightforward way. Indeed, it is not possible to prove any stability property for finite energy solutions to \eqref{eq:MS}. This issue is already thoroughly discussed in the literature, see for instance \cite{AMS, Geyer}, as it seems that there is some technical obstructions in proving a well-posedness result below the $H^{11/8}$ regularity framework for the wave function.
\newline 
For this reason, here we adopt a different strategy. We avoid  to pass through an approximation argument and we rather exploit the integral formulation associated to \eqref{eq:MS}. More precisely, instead of trying to derive the differential equations in \eqref{eq:qmhd} for regular solutions and then passing to the limit into their weak formulation, we directly consider a (finite energy) weak solution to \eqref{eq:MS}, as given by the corresponding Duhamel's formula, and we directly prove that the hydrodynamical quantities defined in \eqref{eq:mad} satisfy the weak formulation associated to \eqref{eq:qmhd}, see the Definition \ref{def:fews} below.
This alternative approach was already introduced in \cite{A_PAMS} in the framework of QHD systems and is somehow similar to the idea used in \cite{Ozawa-2006} in order to prove the conservation in time of the physical quantities associated to the NLS equation, see also \cite{FM}. 
\subsection{Defining the Lorentz force}
In order to rigorously justify all steps and give meaning to all terms in \eqref{eq:qmhd}, some a priori estimates are necessary. Indeed, the self-consistent electromagnetic forces acting on the quantum plasma, determined through Madelung transformations, are a priori not well defined. 
In particular, for weak solutions to \eqref{eq:MS} which are simply in the energy space, we are unable to prove that the Lorentz force \eqref{eq:lor} is locally integrable. 
Our strategy for overcoming these mathematical difficulties
combines hypotheses of extra regularity on initial electromagnetic fields, with
suitable dispersive-type estimates for the weak solutions of \eqref{eq:MS}.
Higher regularity for the electromagnetic fields also plays a role in the analysis of the incompressible Navier-Stokes-Maxwell system done by \cite{Masmoudi, AG, ARIMA} or the incompressible Euler-Maxwell system in \cite{AH}, in order to analyze the Lorentz force provided by the Ohm's law. 
However, we underline that our QMHD \eqref{eq:qmhd} system is irrotational and compressible, therefore it presents several further difficulties in order to deriving the hydrodynamic equations rigorously, even in a weak sense, therefore requires new non-trivial a priori bounds. 
\subsection{Main results}
Before stating our main results, we introduce some notations.
Given $s,\sigma\geq1$, we define the spaces
\begin{gather*}\Sigma^{\sigma}:=\{(A_0,A_1)\in H^{\sigma}(\R^3;\R^3)\times H^{\sigma-1}(\R^3;\R^3)\,s.t.\,\diver A_0=\diver A_1=0\},
\end{gather*}
and
\begin{equation}\label{eq:M_ss}
M^{s,\sigma}:=H^s(\R^3)\times\Sigma^{\sigma}.
\end{equation} 
Moreover, we say that the initial datum $(\rho_0,J_0,E_0,B_0)$ is $(s,\sigma)$-admissible if there exists $(\psi_0,A_0,A_1)\in M^{s,\sigma}$ such that
\begin{equation}\label{eq:ssigma_id}
\rho_0=|\psi_0|^2,\, J_0:=\RE\big(\overline{\psi}_0(-i\nabla-A_0)\psi_0\big),\, E_0=-A_1-\nabla\phi_0,\, B_0:=\nabla\wedge A_0.
\end{equation}
In order to rigorously address arbitrary weak solutions for quantum fluid models,
it turns out that the right hydrodynamic quantities are given by  $\sqrt{\rho}$ and $\Lambda$, where $\Lambda$ is such that $J=\sqrt{\rho}\Lambda$, see Section \ref{sse:polar} and the references \cite{AM_CMP, AM_2D, AM_B, AMZ} for a more exhaustive discussion. The formulation of the problem in terms of $\sqrt{\rho}$ and $\Lambda$ has the advantage of considering well-defined objects for finite energy weak solutions. Furthermore, they can be defined directly, by polar factorization from a given finite energy wave function. 
\begin{theorem}\label{th:main}
Let $\gamma\in(1,3)$, $\sigma>1$, and let $(\rho_0,J_0,E_0,B_0)$ be an $(1,\sigma)$-admissible initial data. Then there exists a global in time, finite energy weak solution $(\rho,J,E,B)$ to the QMHD system \eqref{eq:qmhd}, in the sense of Definition \ref{def:fews}. 
\end{theorem}
\begin{remark}
Let us remark that the critical case $\gamma=3$ in the equation of state for the pressure term requires a finer analysis, due to the necessary a priori estimates to rigorously justify the derivation of fluid dynamical equations. For the sake of clarity in our presentation, we omit here this case and postpone it to a forthcoming paper.
\end{remark}

As it will be clear from the proof of Theorem \ref{th:main}, also the solutions constructed here satisfy a (generalized version of) irrotationality condition, see \eqref{eq:gen_irr} and Remark \ref{rmk:gic}. In this sense such solutions belong to the same framework as those ones studied in \cite{G, GP, GM, GIP} and related papers. 
Nonetheless, we remark that our solutions do carry some vorticity, as quantized vortices are allowed in the vacuum region $\{\rho=0\}$, see \cite{Donn}.
For a more precise discussion on those aspects we refer to Remark \ref{rmk:gic} below.
\newline
For QHD systems (i.e.~system \eqref{eq:qmhd} with no magnetic fields), the well-posedness of the underlying NLS equation, combined with the stability of polar factorization, implies a stability property for a class of finite energy weak solutions to the QHD system, see also \cite{AMZ, AMZ_2D} for more general results that do not rely on the underlying wave function dynamics.
\newline
Due to the strategy of proof adopted for Theorem \ref{th:main}, it is straightforward to see that such stability properties cannot be deduced for the solutions constructed here. However, by exploiting some further delicate smoothing estimates for \eqref{eq:MS}, we can prove the following result.
\begin{theorem}\label{th:stability}
Let $\gamma\in(1, 3)$, $\sigma>1$ and let 
$\{(\psi_0^{(n)},A_0^{(n)},A_1^{(n)})\}\subset M^{1, \sigma}$ be uniformly bounded. Let $(\rho^{(n)}_0,J^{(n)}_0,E^{(n)}_0,B^{(n)}_0)$ be defined by $(\psi_0^{(n)},A_0^{(n)},A_1^{(n)})$ according to identities \eqref{eq:mad}, \eqref{eq:em_f} and let $\big(\rho^{(n)},J^{(n)},E^{(n)},B^{(n)}\big)$ be the weak solution to \eqref{eq:qmhd} with initial data $(\rho^{(n)}_0,J^{(n)}_0,E^{(n)}_0,B^{(n)}_0)$, as constructed in Theorem \ref{th:main}.
\newline
Then there exists a subsequence, not relabeled, and a global in time, finite energy weak solution $(\rho,J,E,B)$ to \eqref{eq:qmhd} such that:
	\begin{itemize}
		\item[(i)] for every $T>0$, $\big(\sqrt{\rho}^{(n)},\Lambda^{(n)},E^{(n)},B^{(n)}\big)$ converges to $(\sqrt{\rho},\Lambda,E,B)$ weak* in $L^{\infty}\Big((0,T);H^1(\R^3)\times \big(L^2(\R^3)\big)^3\Big)$;
		\item[(ii)] by defining the Lorentz force associated to $\big(\rho^{(n)},J^{(n)},E^{(n)},B^{(n)}\big)$ as $F_L^{(n)}:=\rho^{(n)} E^{(n)}+J^{(n)}\wedge B^{(n)}$, then we have $F_L^{(n)}\rightarrow F_L$ in $L^1_{\mathrm{loc}}(\R_t^+\times \R_x^3)$.
	\end{itemize}
\end{theorem}

A key tool in the proof of Theorem \ref{th:stability} will be the local-smoothing estimate in Proposition \ref{pr:smoot_MS} below, which provides sufficient compactness in order to guarantee the weak stability of the hydrodynamic variables. We point out that also the stability of the Lorentz force heavily relies on the extra-regularity condition $\sigma>1$. As mentioned before, an analogous situation occurs in classical magnetohydrodynamic system. In this perspective we mention the paper \cite{ARIMA} on the derivation of the MHD system, where the authors show that even the compensated compactness method fails to prove the weak stability of the Lorentz force in the energy space.

\subsection{Outline of the paper}
The paper is organized as follows. In Section \ref{sec:prel} we introduce some preliminary results we will need throughout the paper, as well as the precise notion of finite energy weak solutions to system \eqref{eq:qmhd}. 
In Section \ref{sec:apr} we prove suitable dispersive estimates for weak solutions to the non-linear Maxwell-Schr\"odinger system \eqref{eq:MS}. 
These estimates play a crucial role in Section \ref{sec:der}, where we rigorously derive the continuity and momentum equations in \eqref{eq:qmhd}. As a byproduct the Lorentz force is well defined.
Section \ref{sec:ls} is devoted to prove non trivial local-smoothing estimates for the Maxwell-Schr\"odinger system \eqref{eq:MS}. Finally, in Section \ref{sec:fin} we prove our main results concerning the QMHD system \eqref{eq:qmhd}.

\section{Preliminaries}\label{sec:prel}
In this section we collect some preliminary results we are going to use throughout the paper and introduce the notation that will be used.
\newline
Given two positive quantities $A,B$, we write $A\lesssim B$ if there exists a constant $C>0$ such that $A\leq CB$. If the constant $C$ depends on a parameter $k$, we write $A\lesssim_k B$. For any $\lambda\in\R$, we set $\langle \lambda\rangle:=\sqrt{1+\lambda^2}$. 
Given a Lebesgue exponent $p\in[1, \infty]$, we denote by $p'$ its dual exponent. For a given vector field $A:\R^3\to\R^3$, we define the magnetic gradient 
\begin{equation*}
\nabla_A:=(-i\nabla-A). 
\end{equation*}
Given $s\in\R$, we write $\mathcal{D}^s:=(1-\Delta)^{s/2}$ for the Bessel operator of order $s$.  We set $\partial_j=\partial_{x_j}$, and we shall occasionally use the repeated indices notations. Given $\alpha\in\Z^3$, we denote by $Q_{\alpha}$ the unit cube centered at $\alpha$ with sides parallel to the axis. 
From now on, we fix a non-negative, smooth function $\chi$ such that $\chi\equiv 1$ on $Q_{0}$ and $\chi\equiv 0$ for $x\in 2Q_0^c$. For $\alpha\in\Z^3$, we set $\chi_{\alpha}(x):=\chi(x-\alpha)$. 
We often write $L^p$ (resp.~$W^{s,p}$) to denote the Lebesgue space $L^p(\R^3)$ (resp.~the Sobolev space $W^{s,p}(\R^3)$). As usual, $H^s$ denotes the space $W^{s,2}$. For any interval $I\subseteq\R$ and any Banach space $\mathcal{X}$, we denote by $L^p(I,\mathcal{X})$ (resp.~$W^{s,p}(I,\mathcal{X})$) the space of $\mathcal{X}$-valued Bochner measurable function on $I$, whose $\mathcal{X}$-norm belongs to $L^p(I)$ (resp.~$W^{s,p}(I)$). These spaces will be often abbreviated to $L_T^p\mathcal{X}$ and $W_T^{s,p}\mathcal{X}$ when $I=[0,T]$. Similarly, we write $\ell_{\alpha}^p\mathcal{X}$ to denote the space $\ell^p(\Z^3,\mathcal{X})$. Given two measurable functions $f,g:\R^3\to\C$, we write
$$(f,g):=\int_{\R^3}\overline{f}gdx,\quad \langle f,g\rangle:=\RE(f,g).$$
When not specified otherwise, $n$ denotes a positive integer constant, which may change at each occurrence.

We state an identity for vector fields in $\R^3$, which we will repeatedly use in the computations.
\begin{equation}\label{eq:tril}
	(V_1\wedge(\nabla\wedge V_2))\cdot V_3=V_1\cdot((V_3\cdot\nabla)V_2)-V_3\cdot((V_1\cdot\nabla)V_2).
\end{equation}

Let us recall the generalized fractional Leibniz rule \cite{Gulisashvili-Kon-1996}.
\begin{lemma}
Let  $s,\,\alpha,\,\beta\in[0,\infty)$, $p\in(1,\infty)$, and let $p_1,\,p_2,\,q_1,\,q_2\in(1,\infty]$ be such that $\frac{1}{p_i}+\frac{1}{q_i}=\frac1p$, $i=1,2$. Then
\begin{equation}\label{frac_leibniz}
\|\mathcal{D}^s(fg)\|_{L^p}\lesssim \|\mathcal{D}^{s+\alpha}f\|_{L^{p_1}}\|\mathcal{D}^{-\alpha}g\|_{L^{q_1}}+\|\mathcal{D}^{-\beta}f\|_{L^{p_2}} \|\mathcal{D}^{s+\beta}g\|_{L^{q_2}}.
\end{equation}
\end{lemma}

In addition we have the following estimate, which can be deduced by the Kato-Ponce commutator estimates \cite{Kato-Ponce_CommEst-1988} and the observation that $\mathbb{P}\nabla=0$, see \cite[Lemma 2.6]{NW1} for details.

\begin{lemma}\label{le:tre}
	Let $s\geq 0$, and let $p,p_1,p_2\in(1,\infty)$, $q_1,q_2\in(1,\infty]$ be such that $\frac{1}{p_i}+\frac{1}{q_i}=\frac1p$, $i=1,2$. Then
	$$\|\mathbb{P}(\bar{f}\,\nabla g)\|_{W^{s,p}}\lesssim\|f\|_{W^{s,p_1}}\|\nabla g\|_{L^{q_1}}+\|\nabla f\|_{L^{q_2}}\|g\|_{W^{s,p_2}}.$$
\end{lemma}

The following lemma (see e.g.~\cite[Proposition 4.9.4]{cazenave}) will be useful when we estimate the pure-power term $|\psi|^{2(\gamma-1)}u$ in fractional Sobolev spaces.

\begin{lemma}\label{esti_pure}
	Let $\gamma>1$, $s\in(0,1)$, $\eps>0$, $1\leq q<p\leq \infty$, and set $\widetilde{p}:=\frac{(2\gamma-1)pq}{p-q}$. Then we have the estimate  
	$$\||\psi|^{2(\gamma-1)}\psi\|_{W^{s,q}}\lesssim\|\psi\|_{L^{\widetilde{p}}}^{2(\gamma-1)}\|\psi\|_{W^{s+\eps,p}}.$$
\end{lemma}
For the Hartree term $\phi\psi$ we will use the bound
\begin{equation}\label{esti_conv}
\|\big((-\Delta)^{-1}|\psi|^2\big)\psi\|_{H^{1/2}}\lesssim\|u\|_{H^1}^3,Ponce commutator estimates
\end{equation}
which follows by combining H\"older and Hardy-Littlewood-Sobolev inequalities -- see e.g.~\cite[Lemma 2.1]{NW2}.

We will also use the inequality
\begin{equation}\label{eq:disin}
	\|\chi_{\alpha} f\|_{\ell_{\alpha}^rW^{s,r}}\lesssim \|f\|_{W^{s,r}},\qquad s\in\R,\,r\in(1,\infty),
\end{equation}
which can be proved analogously as in \cite[Lemma 2.1]{Wada2012}.

Next, we state a version of the Strichartz estimates for the wave equation. We say that a pair $(q,r)$ is wave-admissible if $\frac{1}{q}+\frac{1}{r}=\frac12$, $q\in(2,+\infty]$. We have the following result, see e.g.~\cite{Brenner,Ginibre-Velo_KG,Ginibre-Velo_Strichartz}.
\begin{lemma}\label{le:uno}
Let $T>0$, $\sigma\geq 1$, and let $(q_0,r_0)$ be a wave-admissible pair. For any given $(A_0,A_1)\in\Sigma^{\sigma}$ and $F\in L_{T}^{q'_0}W^{\sigma-1+2/q_0,r'_0}$, there exists a unique solution  $A\in C([0,T],H^{\sigma})\cap C^1([0,T],H^{\sigma-1})$ to the equation $\square A=F$, with initial data $A(0)=A_0$, $\partial_tA(0)=A_1$, which satisfies the estimate
\begin{equation}\label{eq:kg_stri}
\max_{k=0,1}\|\partial_t^k A\|_{L_T^{q}W^{\sigma-k-2/q,r}}\lesssim_T\|(A_0,A_1)\|_{\Sigma^{\sigma}}+\|F\|_{L_T^{q'_0}W^{\sigma+2/q_0-1,r'_0}}
\end{equation}
for every wave-admissible pair $(q,r)$.
\end{lemma}

We collect now a series of dispersive-type estimates for the Schr\"odinger equation. 
First of all, we are going to extensively use the following endpoint Strichartz estimates

\begin{gather}
\label{clasti}	\|e^{it\Delta}f\|_{L_T^2L^6}\lesssim \|f\|_{L_T^2},\\
\label{class_stri}
\big\|\int_0^te^{i(t-s)\Delta}F(s)ds\big\|_{L_T^2L^6}\lesssim \|F\|_{L_T^1L^2+L_T^2L^{6/5}}.
\end{gather}
For more details we refer to \cite{Keel-Tao} and the references therein. Let us further state the following smoothing-Strichartz estimate, which can be deduced by the results in \cite[Lemma 3]{Ionescu-Kenig}.
\begin{equation}\label{sm-du-st}
		\big\|\chi_{\alpha}\,\mathcal{D}^{1/2}\!\int_0^t e^{i(t-s)\Delta}F(s)ds\big\|_{\ell_{\alpha}^{\infty}L_T^2L^2}\lesssim \|F\|_{L_T^1L^2+L_T^2L^{6/5}}.
\end{equation}
We also recall the following Koch-Tzvetkov type-estimate, see e.g.~\cite[Lemma 2.4]{NW1}.
\begin{lemma}\label{le:due}
Let $T>0$ and $s,\alpha\in\R$. Fix moreover $F\in L_T^{2}H^{s-2\alpha}$, and let $\psi\in L^{\infty}_TH^s$ be a weak solution to $i\partial_t \psi=-\Delta \psi + F$. Then $\psi$ satisfies 
\begin{equation}\label{eq:kt}
\|\psi\|_{L_T^2W^{s-\alpha,6}}^2\lesssim T^{-1}\|\psi\|_{L_T^2H^s}^2+T\|F\|_{L_T^2H^{s-2\alpha}}^2.
\end{equation}
\end{lemma}
The Koch-Tzvetkov estimate \eqref{eq:kt} allows in particular to prove a dispersive estimate with $1/2$-loss of derivatives for the magnetic Schr\"odinger flow. Indeed we have the following result \cite[Lemma 2.8]{AMS}.

\begin{lemma}\label{le:quattordici}
	Let $T>0$, $s\in [1,2]$, $\alpha\in[\frac12,1)$, $\sigma\geq 1$, with $(\alpha,\sigma)\neq(\frac12,1)$. Let $A\in L_T^{\infty}H^{\sigma}\cap L_T^2L^{3/(2\alpha-1)}$, with $\diver A=0$, and $F\in L_T^{2}H^{s-2\alpha}$. Then a weak solution $\psi$ to $i\partial_t \psi=-\Delta_A \psi + F$ satisfies
	\begin{equation}\label{ext:nw}
		\|\psi\|_{L_T^2W^{s-\alpha,6}} \lesssim_T \langle\|A\|_{ L_T^{\infty}H^{\sigma}\cap L_T^2L^{3/(2\alpha-1)}}\rangle^n\|\psi\|_{L_T^{\infty}H^s}+\|F\|_{L_T^2H^{s-2\alpha}}.
	\end{equation}
\end{lemma}

Next we recall some useful results for time-independent magnetic potentials. For any given $A\in L^2_{\mathrm{loc}}(\R^3)$, the magnetic Laplacian $-\Delta_A$ can be defined as a non-negative self-adjoint operator on $L^2(\R^3)$, by means of a quadratic form argument \cite{Simon-79_maximal-minimal_Schr_forms}. For a given $\lambda\in\C\setminus(-\infty,0]$, we denote by $R_A(\lambda):=(-\Delta_A+\lambda)^{-1}$ the resolvent of the magnetic Laplacian. Given $s\in\R$, we can define the covariant fractional derivate $\mathcal{D}_A^s:=(1-\Delta_A)^{s/2}$ and the magnetic Sobolev space 
$$H_A^s(\R^3):=\operatorname{Dom}\,\mathcal{D}_A^s,\quad\|f\|_{H_A^s(\R^3)}:=\|\mathcal{D}_A^sf\|_{L^2(\R^3)}.$$
Observe that, as a consequence of the spectral representation of $-\Delta_A$, we have the following result.
\begin{lemma}\label{lessi}
Let $\sigma\in[-2,2]$, $s\in[\sigma,\sigma+2]$ and $\delta>0$. Then we have the estimate
\begin{equation}\label{eq:s-sigma}
\|R_A(\lambda)f\|_{H_A^s}\lesssim \langle\lambda\rangle^{-1+(s-\sigma)/2}\|f\|_{H_A^{\sigma}},\quad |\!\operatorname{arg}\lambda|\leq\pi-\delta.
\end{equation}
\end{lemma}
When the magnetic potential is regular enough, the classical and the magnetic Sobolev norms, for a suitable regime of regularity, are equivalent. In particular, we shall use the following result, see e.g.~\cite[Lemma 2.2]{NW1}, \cite[Lemma 2.5]{Wada2012}.
\begin{lemma}\label{le:quattro}
Assume that $A\in \dot{H}^1(\R^3)$. Then for every $s\in[-2,2]$, we have $H_A^s(\R^3)\cong H^s(\R^3)$. More precisely, the following estimate holds true:
\begin{equation}\label{equicono}
\langle\|A\|_{\dot{H}^1}\rangle^{-n}\|f\|_{H_A^s}\lesssim\|f\|_{H^s}\lesssim \langle\|A\|_{\dot{H}^1}\rangle^{n}\|f\|_{H_A^s}.
\end{equation}
\end{lemma}

\subsection{The polar factorization technique and the rigorous notion of weak solutions}\label{sse:polar}
We now recall some facts about the polar factorization technique. For more details and general results  we refer the reader to \cite{AM_B}  and references therein.
As already discussed in \cite{AM_CMP, AM_2D, AMZ, AHMZ}, in the context of quantum fluid models the hydrodynamic state is suitably described in terms of the pair $(\sqrt{\rho}, \Lambda)$, that in turn will determine the charge and current densities by $\rho=(\sqrt{\rho})^2, J=\sqrt{\rho}\Lambda$, respectively.
This is promptly motivated by the fact that the total energy defined in \eqref{eq:en} provides natural bounds for $(\sqrt{\rho}, \Lambda)$, see also \eqref{eq:tot_ene} below.
Moreover, the polar factorization approach allows to define the hydrodynamic state $(\sqrt{\rho}, \Lambda)$ almost everywhere in the whole space. Given a function $\psi\in H^1(\R^3;\C)$, we define the set of its polar factors by
\begin{equation}\label{eq:pol}
P(\psi):=\{\varphi\in L^{\infty}(\R^3)\,s.t.\,
\|\varphi\|_{L^{\infty}}\leq 1,\,\psi=|\psi|\varphi\mbox{ a.e. in }\R^3\}.
\end{equation}
Observe that $\varphi$ is uniquely determined $|\psi|dx$ almost everywhere, while it is not uniquely defined in the nodal region $\{\psi=0\}$. We have the following result \cite[Lemma 5.3]{ADM}.

\begin{lemma}\label{le:polar}
Let $\psi\in H^1(\R^3)$, $A\in L^3(\R^3)$, $\varphi\in P(\psi)$, and let us define 
\begin{equation*}
\sqrt{\rho}=|\psi|, \quad \Lambda=\RE(\bar{\varphi}\nabla_A\psi).
\end{equation*}
Then we have
\begin{itemize}
\item[(i)] $\sqrt{\rho}\in H^1(\R^3)$, with $\nabla\sqrt{\rho}=-\IM(\bar\varphi\nabla_A\psi)=\RE(\bar{\varphi}\nabla\psi)$;
\item[(ii)] the following identity holds almost everywhere on $\R^3$:
\begin{equation}\label{eq:id-pol}
\RE\left(\overline{\nabla_A\psi}\otimes\nabla_A\psi\right)
=\Lambda\otimes\Lambda+\nabla\sqrt{\rho}\otimes\nabla\sqrt{\rho};
\end{equation}
\item[(iii)] by defining $J=\sqrt{\rho}\Lambda$, we have that the following identity
\begin{equation*}
    \nabla\wedge J+\rho\nabla\wedge A=2\nabla\sqrt{\rho}\wedge\Lambda
\end{equation*}
holds in the distributional sense.
\end{itemize}
Furthermore, let $\{\psi_n\}\subseteq H^1(\R^3)$, $\{A_n\}\subseteq L^3(\R^3)$ be such that $\psi_n\rightarrow \psi$ in $H^1$ and $A_n\rightarrow A$ in $L^3$. 
Let moreover $\varphi_n\in P(\psi_n)$, and set $\sqrt{\rho}_n:=|\psi_n|$, $\Lambda_n:=\IM(\bar{\varphi_n}\nabla_{A_n}\psi_n)$. 
Then we have
\begin{equation}\label{eq:ssp}
\sqrt{\rho}_n\rightarrow\sqrt{\rho}\mbox{ in }H^1(\R^3),\quad \Lambda_n\rightarrow \Lambda\mbox{ in }L^2(\R^3).
\end{equation}
\end{lemma}

Let us notice that, 
by the definition of the hydrodynamic state $(\sqrt{\rho}, \Lambda)$ as given in previous Lemma, the identity for the current density $J$ in (iii) above is consistent with the Madelung transformation \eqref{eq:mad}. Moreover, by Theorem 6.19 in \cite{LL}, we have that $\nabla\psi=0$ a.e. in $\{\psi=0\}$ and consequently $\Lambda=0$ a.e. on $\{\rho=0\}$. This implies that the convective term in the equation for the momentum density may be written as
\begin{equation*}
    \frac{J\otimes J}{\rho}=\Lambda\otimes\Lambda.
\end{equation*}
Analogously the total energy in \eqref{eq:en} may be expressed as
\begin{equation}\label{eq:tot_ene}
\mathcal E(t)=\int_{\R^3}\frac12|\nabla\sqrt{\rho}|^2+\frac12|\Lambda|^2+f(\rho)+\frac12|E|^2+\frac12|B|^2\,dx.
\end{equation}
By using identity \eqref{eq:id-pol}, the definitions \eqref{eq:em_f} and the fact that $\diver A=0$ in \eqref{eq:MS}, we can also show that the energy functional in \eqref{eq:tot_ene} actually equals the total energy associated to system \eqref{eq:MS}, given by
\begin{equation}\label{eq:tot_ene_m}
\mathcal E(t)  = 
\int_{\R^3}\frac12|(-i\nabla-A)\psi|^2+f(|\psi|^2)+\frac12|\d_tA|^2+\frac12|\nabla A|^2\,dx.
\end{equation}
Consequently, thanks to the polar factorization, there will be no ambiguity in using the same notation either for the total energy associated the fluid system \eqref{eq:qmhd} and the one for the wave dynamics \eqref{eq:MS}.
\newline
Furthermore, let us notice that the quantum pressure term appearing in the equation for the current density in \eqref{eq:qmhd} can also be written in the following form,
\begin{equation}\label{eq:bohm}
\frac12\rho\nabla\left(\frac{\Delta\sqrt{\rho}}{\sqrt{\rho}}\right)
=\frac14\nabla\Delta\rho-\diver(\nabla\sqrt{\rho}\otimes\nabla\sqrt{\rho}).
\end{equation}
Consequently, it is possible to rewrite the equation for the current density in the following way
\begin{equation}\label{eq:mom_w}
    \d_tJ+\diver\left(\Lambda\otimes\Lambda+\nabla\sqrt{\rho}\otimes\nabla\sqrt{\rho}\right)+\nabla P(\rho)=\rho E+J\wedge B+\frac14\nabla\Delta\rho.
\end{equation}

Let us notice that, in this form, all terms appearing in \eqref{eq:mom_w} are well-defined for arbitrary finite energy states, at least in the sense of distribution, with the exception of the magnetic force $J\wedge B$, that in general is not known to be locally integrable.

\begin{definition}\label{def:fews}
Let $0<T\leq\infty$, $\rho_0, J_0, E_0, B_0\in L^1_{loc}(\R^3)$, such that $\diver E_0 =\rho_0$, $\diver B_0=0$ in the distributional sense and
\begin{equation*}
\mathcal E_0:=
\int_{\R^3}\frac12|\nabla\sqrt{\rho_0}|^2+\frac12\frac{|J_0|^2}{\rho_0}
+f(\rho_0)+\frac12|E_0|^2+\frac12|B_0|^2\,dx<\infty.
\end{equation*} 
We say that $(\rho, J, E, B)$ is a finite energy weak solution to the Cauchy problem \eqref{eq:qmhd}-\eqref{eq:qmhd_id} on the space-time slab $[0,T)\times\R^3$ if the following conditions are satisfied.
\begin{itemize}
\item[(i)] (hydrodynamic state) there exists $(\sqrt{\rho}, \Lambda)$, with $\sqrt{\rho}\in L_{\mathrm{loc}}^2(0, T;H^1(\R^3))$ and $\Lambda\in L_{\mathrm{loc}}^2(0,T;L^2(\R^3))$, such that $\rho=(\sqrt{\rho})^2$, $J=\sqrt{\rho}\Lambda$;
\item[(ii)] (electromagnetic field) $E,B\in L^2_{\mathrm{loc}}((0,T);L^2(\R^3))$;
\item[(iii)] (Lorentz force) $F_L:=\rho E+J\wedge B\in L^1_{\mathrm{loc}}((0,T)\times \R^3)$;
\item[(iv)] (continuity equation) for every $\eta\in\mathcal{C}^{\infty}_c([0,T)\times\R^3)$,
\begin{equation}\label{eq:continuity}
\int_0^T\int_{\R^3}\rho\partial_t\eta+J\cdot\nabla\eta\,dxdt+\int_{\R^3}\rho_0(x)\eta(0,x)dx=0;
\end{equation}
\item[(v)] (momentum equation) for every $\zeta\in\mathcal{C}^{\infty}_c([0,T)\times\R^3;\R^3)$,
\begin{equation}\label{eq:momentum}
\begin{split}
\int_0^T\int_{\R^3}J&\partial_t\zeta+\nabla\zeta:\big(\nabla\sqrt{\rho}\otimes\nabla\sqrt{\rho}+\Lambda\otimes\Lambda\big)+(\rho E+J\wedge B)\cdot\zeta\\
&-\frac{1}{4}\rho\Delta\diver\zeta+P(\rho)\diver\zeta dxdt+\int_{\R^3}J_0(x)\cdot\zeta(0,x)dx=0;
\end{split}
\end{equation}
\item[(vi)] (Maxwell equations) the equations
\begin{equation}\label{eq:Max}
\begin{cases}
\diver E=\rho,\quad\nabla\wedge E=-\partial_t B\\
\diver B=0,\quad\nabla\wedge B=J+\partial_t E,
\end{cases}
\end{equation}
are satisfied in the sense of distributions;
\item[(vii)] (finite energy) the total energy defined in \eqref{eq:tot_ene} is finite for a.e. $t\in[0, T)$ and we have $\mathcal E(t)\leq \mathcal E_0$;
\item[(viii)] (generalized irrotationality condition) for a.~e.~$t\in(0,T)$,
\begin{equation}\label{eq:gen_irr}
\nabla\wedge J+\rho B=2\nabla\sqrt{\rho}\wedge\Lambda
\end{equation}
holds true in $\mathcal{D}'(\R^3)$;
\end{itemize}
If we can take $T=\infty$ we say that the solution is global.
\end{definition}

\begin{remark}\label{rmk:gic}
Let us consider a sufficiently smooth solution, for which it is also possible to define the velocity field $v$ such that $J=\rho v$, then it is straightforward to check that condition \eqref{eq:gen_irr} is equivalent to the following identity
\begin{equation*}
\rho\left(\nabla\wedge v+B\right)=0.
\end{equation*}
Hence for smooth solutions with a nowhere vanishing mass density, \eqref{eq:gen_irr} is equivalent to the (scaled) irrotationality condition assumed in \cite{GM}. Thus, in some generalized sense, our framework is compatible with the one studied in \cite{GM}, see also \cite{G, GP, GIP}. On the other hand, let us stress that the solutions under our consideration here do carry vorticity, which may be concentrated in the vacuum region $\{\rho=0\}$. 
Consequently, our framework embodies in the theory the presence of quantized vortices, which are relevant in the study of quantum fluids \cite{TKT}.
\end{remark}
\subsection{Weak solutions to the nonlinear Maxwell-Schr\"odinger system}
We conclude this section by introducing some notations related to weak solutions for the nonlinear Maxwell-Schr\"odinger system \eqref{eq:MS}. 

Let $T>0$ and $(\psi_0, A_0, A_1)\in M^{1, 1}$, see definition \eqref{eq:M_ss}. We say that $(\psi, A)$ is a weak $M^{1, 1}$-solution to \eqref{eq:MS} on $[0, T]\times\R^3$ with initial data $(\psi_0, A_0, A_1)$ if
\begin{equation*}
(\psi,A)\in \big(L_T^{\infty}H^1\cap W_T^{1,\infty}H^{-1}\big)\times\big(L_T^{\infty}H^{1}\cap W_T^{1,\infty}L^2\cap W_T^{2,\infty}H^{-1}\big),
\end{equation*}
the two equations in \eqref{eq:MS} are satisfied in $H^{-1}$ for a.e.~$t\in[0, T]$, and we have 
$(\psi(0), A(0), \d_tA(0))=(\psi_0, A_0, A_1)$. For the sake of conciseness, we are going to use the notation $(\psi, A)\in L^\infty_TM^{s, \sigma}$ to actually mean that 
$(\psi, A, \d_tA)\in L^\infty_TM^{s, \sigma}$.

The existence of global in time, weak $M^{1, 1}$-solutions to \eqref{eq:MS} can be obtained by a standard regularization approach that passes through the analysis of an approximating system and suitable compactness estimates for the approximating solutions. In particular, in \cite{ADM} a Yosida-type regularization method was used -- see also \cite{GNS,AMS-2017-globalFinErgNLS}, where a vanishing viscosity method is used to study respectively the Maxwell-Schr\"odinger system and the NLS equation with critical, time-dependent magnetic potentials. For an overview of local and global well-posedness results for the Maxwell-Schr\"odinger system and its nonlinear variant \eqref{eq:MS}, we refer to \cite{NT, Tsu, NW1, NW2, BT, AMS, Geyer}

We state the following result, which can be proved by suitably adapting the analysis done in \cite{ADM}.
\begin{proposition}\label{pr:exMS}
Given $(\psi_0, A_0, A_1)\in M^{1, 1}$, there exists a global in time, weak $M^{1,1}$-solution $(\psi, A)$ to \eqref{eq:MS} with initial data $(\psi_0, A_0, A_1)$, satisfying
$$\|(\psi,A,\partial_t A)\|_{L_T^{\infty}M^{1,1}}\lesssim_T \|(\psi_0,A_0,A_1)\|_{M^{1,1}}$$
for every $T>0$, and such that $\mathcal E(t)\leq \mathcal E(0)$ for a.e. $t>0$, where the energy $\mathcal E$ is defined in \eqref{eq:tot_ene_m}.
\end{proposition}

\section{A priori estimates for weak solutions to the Maxwell-Schr\"odinger system}
\label{sec:apr}
In this section we prove suitable a priori estimates for solutions to the nonlinear Maxwell-Schr\"odinger system \eqref{eq:MS}. Our main goal here is to show the following uniform bounds for weak $M^{1,1}$-solutions, assuming the extra-regularity assumption $\sigma>1$ on the initial magnetic potential.

\begin{proposition}\label{th:apriori}
	Let $\gamma\in(1,3)$, $\sigma\in(1,\frac76)$, and let $(\psi,A)$ be a global weak $M^{1,1}$-solution to the system \eqref{eq:MS}. If we further assume that the initial data satisfy $(\psi_0,A_0,A_1)\in M^{1,\sigma}$, then the following estimate
	\begin{equation}\label{eq:apriori_uno}
\|\psi\|_{L_T^2W^{1/2,6}}+\|A\|_{L_T^{\infty}H^{\sigma}\cap L_T^2L^{\infty}}
\leq C(T,\|(\psi,A)\|_{L_T^{\infty}(H^1\times H^1)}, \|(A_0, A_1)\|_{\Sigma^\sigma}),
\end{equation}
holds for any $T\in(0,\infty)$.
\end{proposition}

The gain of local integrability for $\psi$ and the persistence of regularity for $A$, provided by estimate \eqref{eq:apriori_uno}, will be crucial in order to justify the derivation of the QMHD system \eqref{eq:qmhd}.

\begin{remark}
Proposition \ref{th:apriori} was already partly proved in \cite[Proposition 3.1]{AMS}, in the case when $\gamma\in(1,\frac52)$. Let us point out that in \cite{AMS} a different convention for $\gamma$ is adopted, as the non-linear term has the form $|\psi|^{\gamma-1}\psi$ therein.
\end{remark}

\begin{proof}[Proof of Proposition \ref{th:apriori}]
As the case $\gamma\in(1, \frac52)$ was already proved in \cite{AMS}, see the previous remark, here we focus on the case $\gamma\in[\frac52, 3)$. First of all, 
we claim that for any wave-admissible pair $(q,r)$ we have
\begin{equation}\label{eq:cirga}
	\|A\|_{L_T^qL_x^r}\lesssim_{T} \langle\|\psi\|_{L_T^{\infty}H^1}^2\rangle\langle\|A\|_{L_T^{\infty}H^1}\rangle + \|(A_0,A_1)\|_{\Sigma^1}.
\end{equation}
This can be proved by adapting the proof of estimate (3.6) in \cite{AMS}, and essentially follows from the Strichartz estimate \eqref{eq:kg_stri}. Next, we claim that for any $\delta\in(0,\frac12)$ we have
\begin{equation}\label{eq:quasi}
	\|\psi\|_{L_T^2W^{1/2-\delta,6}}\lesssim_{T}\langle\|(\psi,A)\|_{L_T^{\infty}(H^1\times H^1)}^n\rangle\langle\|(A_0,A_1)\|_{\Sigma^{\sigma}}^n\rangle.
\end{equation}
Let us denote by $(q_{\delta},r_{\delta})$ the wave-admissible pair with $r_{\delta}=\frac{3}{2\delta}$. Moreover, let us set $\delta_1=\delta_1(\gamma):=\gamma-\frac52\in[0,\frac12)$ and $p=p(\gamma):=\frac{6}{7-2\gamma}$. Then the following chain of inequalities may be obtained
\begin{equation}\label{ba_db}
	\begin{split}
		\|\psi\|_{L_T^2W^{1/2-\delta_1-\delta,6}}&\lesssim_T\langle\|A\|_{L_T^{\infty}H^{1}\cap L_T^{q_{\delta}}L^{r_{\delta}}}^n\rangle\|\psi\|_{L_T^{\infty}H^1}\\
		&\quad+\|\phi\psi\|_{L_T^2L^2}+\||\psi|^{2(\gamma-1)}\psi\|_{L_T^2W^{1/2-\delta_1-\delta,6/5}}\\
		&\lesssim_{T}\langle\|(\psi,A)\|_{L_T^{\infty}(H^1\times H^1)}^n\rangle\langle\|(A_0,A_1)\|_{\Sigma^{\sigma}}^n\rangle\\
		&\quad+\|\psi\|_{L_T^{\infty}H^1}^3 +\|\psi\|_{L_T^{\infty}L^6}^{2(\gamma-1)}\|\psi\|_{L_T^{\infty}W^{(1/2-\delta_1)-,p}},
	\end{split}
\end{equation}
where we used the magnetic Koch-Tzvetkov estimate \eqref{ext:nw} and the inhomogeneous Strichartz estimate \eqref{class_stri} in the former inequality, and the bounds \eqref{eq:cirga} and \eqref{esti_conv} together with Lemma \ref{esti_pure} in the latter. Combining estimate \eqref{ba_db} with the Sobolev embedding $H^1\hookrightarrow W^{1/2-\delta_1,p}$ we obtain
\begin{equation}\label{ba_dbu}
\|\psi\|_{L_T^2W^{1/2-\delta_1-\delta,6}}\lesssim_{T}\langle\|(\psi,A)\|_{L_T^{\infty}(H^1\times H^1)}^n\rangle\langle\|(A_0,A_1)\|_{\Sigma^{\sigma}}^n\rangle.
\end{equation}
When $\gamma=\frac52$, \eqref{ba_dbu} directly boils down to the desired estimate \eqref{eq:quasi}, for $\delta_1(\frac52)=0$. Next we consider the case $\gamma\in(\frac52,2+\frac{1}{\sqrt{2}}]$. Arguing as in \eqref{ba_db} we get
\begin{equation*}
	\begin{split}
		\|\psi\|_{L_T^2W^{1/2-,6}}&\lesssim_{T}\langle\|(\psi,A)\|_{L_T^{\infty}(H^1\times H^1)}^n\rangle\langle\|(A_0,A_1)\|_{\Sigma^{\sigma}}^n\rangle\\
		&\quad+\|\psi\|_{L_T^{\infty}L^6}^{2(\gamma-1)}\|\psi\|_{L_T^{2}W^{1/2-,p}}.
	\end{split}
\end{equation*}
The bound above, together with the embedding $H^1\cap W^{1/2-\delta_1,6}\hookrightarrow W^{1/2,p}$ and estimate \eqref{ba_dbu},  guarantees the validity of \eqref{eq:quasi}. In the remaining case $\gamma\in(2+\frac{1}{\sqrt{2}},3)$, we first show the boundedness of $\|\psi\|_{L_T^2W^{(1/2-\delta_2)-,6}}$ for some $\delta_2\in(0,\delta_1)$, then by iterating that same argument for a sufficient number of times, we eventually deduce estimate \eqref{eq:quasi}. Next, using \eqref{eq:quasi} and arguing as in the proof of estimate (3.1) in \cite{AMS}, we obtain the bound
\begin{equation}\label{eq:apa}
	\begin{split}
		&\|A\|_{L_T^2L^{\infty}}+\|(A,\partial_t A)\|_{L_T^{\infty}\Sigma^{\sigma}}\lesssim_{T}\langle\|(\psi,A)\|_{L_T^{\infty}(H^1\times H^1)}^n\rangle\langle\|(A_0,A_1)\|_{\Sigma^{\sigma}}^n\rangle.
	\end{split}
\end{equation}
We apply now Lemma \ref{le:quattordici} with $\alpha=1/2$, the Strichartz estimate \eqref{class_stri}, the bounds \eqref{eq:apa},\eqref{esti_conv} and Lemma \ref{esti_pure}, to obtain
\begin{equation*}
	\begin{split}
		\|\psi\|_{L_T^2W^{1/2,6}}&\lesssim_{T}\langle\|A\|_{L_T^{\infty}H^{\sigma}\cap L_{T}^2L^{\infty}}^n\rangle\|\psi\|_{L_T^{\infty}H^1}+\|\phi \psi\|_{L_T^2L^2}\\
		&\quad +\||\psi|^{2(\gamma-1)}\psi\|_{L_T^2W^{1/2,6/5}}\\
		&\lesssim_{T}\langle\|(\psi,A)\|_{L_T^{\infty}(H^1\times H^1)}^n\rangle\langle\|(A_0,A_1)\|_{\Sigma^{\sigma}}^n\rangle\\
		&\qquad +\|\psi\|_{L_T^{\infty}L^6}^{2(\gamma-1)}\|\psi\|_{L_T^2W^{1/2+\eps,p}},
	\end{split}
\end{equation*}
for every $\eps>0$. Let us choose $\eps,\delta>0$ small enough such that we have the embedding $H^1\cap W^{1/2-\delta,6}\hookrightarrow W^{1/2+\eps,p}$ (this choice is indeed always possible, as $p(\gamma)\in[3,6)$ for $\gamma\in[\frac52,3)$). The bound above combined with \eqref{eq:quasi} then yields the dispersive estimate 
$$\|\psi\|_{L_T^2W^{1/2,6}}\lesssim_{T}\langle\|(\psi,A)\|_{L_T^{\infty}(H^1\times H^1)}^n\rangle\langle\|(A_0,A_1)\|_{\Sigma^{\sigma}}^n\rangle,$$
which together with \eqref{eq:apa} gives the desired a priori estimate \eqref{eq:apriori_uno}.
\end{proof}

The a priori bounds proved in Proposition \ref{th:apriori} in particular imply that, for 
$M^{1, \sigma}$ initial data, $\sigma\in(1,\frac76)$, the solution preserves this regularity globally in time. 
Moreover, as a consequence of Proposition \ref{th:apriori} we also have that, by defining the hydrodynamic variables as usual, the associated Lorentz force $F_L:=\rho E + J\times B$ is well defined. More precisely, analogously to \cite[Proposition 3.2]{AMS}, we can show the following result.

\begin{proposition}\label{pr:wdlf}
Let the same assumptions of Proposition \ref{th:apriori} be satisfied, then $F_L\in L_T^2L^{1}$.
\end{proposition}

\section{Derivation of the continuity and momentum equation}\label{sec:der}
We now turn our attention to rigorously derive the hydrodynamic equations. Given $\sigma>1$ we consider a global, weak $M^{1,1}$-solution $(\psi,A)$ to the system \eqref{eq:MS}, with initial data $(\psi_0,A_0,A_1)\in M^{1,\sigma}$. 

By exploiting the a priori bound provided by Proposition \ref{th:apriori}, we are going to show that the hydrodynamic variables $(\rho, J, E, B)$, defined by \eqref{eq:mad}-\eqref{eq:em_f}, satisfy the weak formulations \eqref{eq:continuity} and \eqref{eq:momentum}. To this aim, we first recall that the integral formulation for the NLS equation in \eqref{eq:MS}, given by the Duhamel's formula 
\begin{equation}\label{eq:duh}
	\psi(t)=e^{\frac{i}{2}t\Delta}\psi_0-i\int_0^te^{\frac{i}{2}(t-s)\Delta}F(s)\,ds,
\end{equation}
holds as an identity in $H^{-1}$ for every $t\geq 0$ 
and in $H^1$ for a.e.~$t\geq 0$, where we set
\begin{equation}\label{eq:def_VF}
	F:=iA\cdot\nabla\psi+\big({\textstyle{\frac12}}|A|^2+\phi+|\psi|^{2(\gamma-1)}\big)\psi.
\end{equation}
Moreover, by standard arguments (see Chapter IV.2.4 in \cite{Sohr} for instance), it is also possible to show that, upon redefining the weak solution on a set of measure zero in 
$[0, T)$, we have that $\psi:[0, T)\to H^1(\R^3)$ is weakly continuous.
\newline 
Consider moreover a test function $\eta\in\mathcal C^{\infty}_c([0, T)\times\R^3)$, for some $T>0$. Since 
$(\psi, A)\in L^\infty_TM^{1, 1}$ satisfy
\begin{equation*}
\big(i\d_t+\frac12\Delta\big)\psi=F,
\end{equation*}
in $H^{-1}$ for a.e. $t\in[0, T]$, then $(\eta\psi, A)\in L^\infty_TM^{1, 1}$ satisfy
\begin{equation}\label{id:feta}
	(i\partial_t+{\textstyle{\frac12\Delta)}}(\eta\psi)=i\psi\partial_t\eta+F_{\eta},
	\end{equation}
where we defined
\begin{equation}\label{eq:F_eta}
	F_{\eta}:=\eta F+{\textstyle{\frac12\psi\Delta\eta}}+\nabla\eta\cdot\nabla\psi,
\end{equation}
with $F$ given by \eqref{eq:def_VF}. 

Next we collect a series of useful estimates for $F$ and $F_{\eta}$, which will be crucial in order to rigorously justify the exchanges of integration order appearing  through the derivation of (the weak formulations of) the continuity and momentum equation. 

\begin{lemma}
Let $\gamma\in(1, 3)$, $\sigma>1$, and $0<T<\infty$. Let $(\psi, A)$ be a weak $M^{1,1}$-solution to \eqref{eq:MS} in $[0, T]\times\R^3$, with initial data $(\psi_0, A_0, A_1)\in M^{1, \sigma}$. Let $F, F_\eta$ be defined as in \eqref{eq:def_VF}, \eqref{eq:F_eta}, respectively. Then we have
\begin{align}
\label{eq:est_V}
\|F\|_{L_T^2L^{3/2}}\lesssim_{T,\sigma}\, &1;\\
\label{eq:esfet}
\|F_{\eta}\|_{L_T^2L^{3/2}}\lesssim_{T,\sigma}\, &1;
\end{align}
Moreover, for any $t_1\in[0, T]$, we have
\begin{align}
\label{fub}
\Big\|\int_{t_1}^{t}\|e^{i(t-\tau)\Delta}F(\tau)\|_{L^3}\,d\tau\Big\|_{L_t^2(0,T)}
\lesssim_{T,\sigma}\, &1;\\
\label{fubeta}
	\Big\|\int_{t_1}^{t}\|e^{i(t-\tau)\Delta}F_{\eta}(\tau)\|_{L^3}\,d\tau
\Big\|_{L_t^2(0,T)}
\lesssim_{T,\sigma}\, &1.
\end{align}
\end{lemma}
In particular, estimate \eqref{eq:esfet} yields $F_\eta\in L^1_TH^{-1/2}$, so that standard arguments on the theory of linear semigroups (see e.g.~\cite[Section 1.6]{cazenave}) guarantee the validity of Duhamel's formula
\begin{equation}\label{duha_eta}
	(\eta\psi)(t_2)=e^{\frac{i}{2}(t_2-t_1)\Delta}(\eta\psi)(t_1)+\int_{t_1}^{t_2}e^{\frac{i}{2}(t_2-t)\Delta}(\psi\partial_t\eta-iF_{\eta})(t)dt,
\end{equation}
for every $t_1,t_2\in[0,T]$.
\begin{proof}
Estimate \eqref{eq:est_V} follows from
\begin{equation*}
	\begin{split}
		\|F\|_{L_T^2L^{3/2}}&\lesssim_T \|A\|_{L_T^{\infty}L^{6}}\|\nabla\psi\|_{L_T^{\infty}L^2}\\
		&+\big(\|A\|_{L_T^{\infty}L^6}^2+\|\psi\|^2_{L_T^{\infty}L^2}\big)\|\psi\|_{L_T^{\infty}L^{3}}+\|\psi\|^{2\gamma-1}_{L_T^{4\gamma-2}L^{3\gamma-3/2}}\\
		&\lesssim_T 1+\|\psi\|^{2\gamma-1}_{L_T^{\infty}H^1\cap L_T^2W^{1/2,6}}\lesssim_{T,\sigma}\, 1,
	\end{split}
\end{equation*}
where we used the definition \eqref{eq:def_VF} of $F$ and H\"older inequality in the first step, Sobolev embedding together with the interpolation inequality $\|\psi\|_{L_T^{4\gamma-2}L^{3\gamma-3/2}}\lesssim_T \|\psi\|_{L_T^{\infty}H^1\cap L_T^2W^{1/2,6}}$ in the second step, and the a priori bound \eqref{eq:apriori_uno} in the last step. Moreover, using \eqref{eq:est_V} and the fact that $\eta$ is smooth and compactly supported, we also get 
\begin{equation*}
	\|F_{\eta}\|_{L_T^2L^{3/2}}\lesssim_T \|F\|_{L_T^2L^{3/2}}+\|\psi\|_{L_T^{\infty}L^2}+\|\nabla\psi\|_{L_T^{\infty}L^2}\lesssim_{T,\sigma}\, 1,
\end{equation*}
which proves \eqref{eq:esfet}. In order to prove \eqref{fub}, we observe that
\begin{equation*}
\Big\|\int_{t_1}^{t}\|e^{i(t-\tau)\Delta}F(\tau)\|_{L^3}\,d\tau\Big\|_{L_t^2(0,T)}\lesssim_T \|F\|_{L_T^{2}L^{3/2}}\lesssim_{T,\sigma}\, 1,
\end{equation*}
where the first inequality can be proved as in the case of the inhomogeneous non-endpoint Strichartz estimates for the Schr\"odinger equation (by combining the fixed-time dispersive bound, fractional integration in the time variable and the Christ-Kiselev Lemma), and the second inequality is guaranteed by \eqref{eq:est_V}. Analogously, using \eqref{eq:esfet} we get 
\begin{equation*}
\Big\|\int_{t_1}^{t}\|e^{i(t-\tau)\Delta}F_{\eta}(\tau)\|_{L^3}\,d\tau\Big\|_{L_t^2(0,T)}\lesssim_T \|F_{\eta}\|_{L_T^{2}L^{3/2}}\lesssim_{T,\sigma}\, 1,
\end{equation*}
which proves \eqref{fubeta} and concludes the proof.
\end{proof}

Let us study separately the continuity and the momentum equation.

\subsection{Continuity equation} We have the following result.
\begin{proposition}\label{pr:continuity}
	Let us fix $\gamma\in(1,3)$, $\sigma>1$, $T>0$, and let $(\psi,A)$ be a weak $M^{1,1}$-solution to the system \eqref{eq:MS} on the space-time slab $[0,T]\times\R^3$, with initial data $(\psi_0,A_0,A_1)\in M^{1,\sigma}$. Then the integral identity
	\begin{equation}\label{eq:cont}
		\int_0^T\int_{\R^3}\rho\partial_t\eta+J\cdot\nabla\eta\,dxdt+\int_{\R^3}\rho_0(x)\eta(0,x)dx=0
	\end{equation}
	holds true for every $\eta\in\mathcal{C}_c^{\infty}([0,T)\times\R^3)$.
\end{proposition}

\begin{proof}
By using the Duhamel formula \eqref{eq:duh}, we write
	\begin{equation}\label{eq:prel_cont}
		\begin{split}
			\int_0^T&\int_{\R^3}\rho\partial_t\eta dxdt=\int_0^T\<\psi\partial_t\eta,\psi\>dt=\int_0^T\<\psi\partial_t\eta,e^{\frac{i}{2}t\Delta}\psi_0\> dt\\
			&-\int_0^T\<\psi\partial_t\eta,\,i\!\int_0^te^{\frac{i}{2}(t-s)\Delta}F(s)ds\>dt.
		\end{split}
	\end{equation}
Estimates \eqref{fub} allows to apply Fubini Theorem, which combined with the unitarity of the propagator $e^{iT\Delta}$ yields
	\begin{equation}\label{matg}
		\begin{split}
			\int_0^T\int_{\R^3}\rho&\partial_t\eta dxdt=\Big\langle\int_0^Te^{\frac{i}{2}(T-t)\Delta}(\psi\partial_t\eta)(t)dt,e^{\frac{i}{2}T\Delta}\psi_0\Big\rangle\\
			&-\int_{\mathcal{D}}\Big\langle e^{\frac{i}{2}(s-t)\Delta}(\psi\partial_t\eta),iF(s)\Big\rangle dsdt:=\mbox{I}+\mbox{II},
		\end{split}
	\end{equation}
where $\mathcal{D}:=\{(s,t)\in[0,T]^2\,|\,0\leq s\leq t\leq T\}$. Let us evaluate terms I and II separately.

	\textbf{Term I.} Since $\operatorname{supp}(\eta)\subseteq [0,T)\times\R^3$, the Duhamel formula \eqref{duha_eta} with $t_1=0$, $t_2=T$ yields
\begin{equation}\label{van_att}
\int_{0}^{T}e^{\frac{i}{2}(T-t)\Delta}(\psi\partial_t\eta)(t)dt = 
-e^{\frac{i}{2}T\Delta}(\eta_0\psi_0)+
i\int_{0}^{T}e^{\frac{i}{2}(T-t)\Delta}F_{\eta}(t)dt.
\end{equation}
By plugging \eqref{van_att} in the definition of $I$ and by exploiting the unitarity of $e^{it\Delta}$, we obtain
	\begin{equation}\label{eq:I}
		\begin{split}
			I&=\Big\langle -e^{\frac{i}{2}T\Delta}\big(\eta_0\psi_0\big)+i\int_0^Te^{\frac{i}{2}(T-t)\Delta}F_{\eta}(t)dt,e^{\frac{i}{2}T\Delta}\psi_0\Big\rangle\\
			&=-\int_{\R^3}\eta(0,x)\rho_0(x)+\int_0^T\big\langle iF_{\eta},e^{\frac{i}{2}t\Delta}\psi_0\big\rangle dt,
		\end{split}
	\end{equation}
	where in the last step we use Fubini Theorem to exchange the order of integration, owing to estimate \eqref{fubeta} and the condition $\operatorname{supp}(\eta)\subseteq [0,T)\times\R^3$.
	
	\textbf{Term II.} We start by writing 
	\begin{equation}\label{eq:102}
		II=\int_0^T\big\langle\int_s^Te^{\frac{i}{2}(s-t)\Delta}(\psi\d_t\eta)(t)\,dt,-iF(s)\big\rangle\,ds.
	\end{equation}
	Applying identity \eqref{duha_eta} with $t_1=T$, $t_2=s$, since $\eta(T, \cdot)\equiv0$, we obtain
	\begin{equation*}
		\int_s^Te^{\frac{i}{2}(s-t)\Delta}(\psi\d_t\eta)(t)\,dt=-(\eta\psi)(s)+i\int_s^Te^{\frac{i}{2}(s-t)\Delta}F_\eta(t)\,dt.
	\end{equation*}
	By plugging the previous identity into \eqref{eq:102} we then deduce
	\begin{multline}\label{eq:I_2}
		II=\int_0^T\langle-(\eta\psi)(s)
		+i\int_s^Te^{\frac{i}{2}(s-t)\Delta}F_\eta(t)\,dt,-iF(s)\rangle\,ds\\
		=\int_0^T\langle \eta\psi, iF\rangle+\langle iF_\eta(t),-i\int_0^te^{\frac{i}{2}(t-s)\Delta}F(s)\,ds\rangle\,dt,
	\end{multline}
	where in the last step we used Fubini Theorem to exchange the order of integration, owing to estimates \eqref{eq:esfet} and \eqref{fub}. 
	
	\textbf{I+II.} By summing the two contributions \eqref{eq:I} and \eqref{eq:I_2} we obtain
	\begin{equation}\label{eq:con_par}
		\int_0^T\int_{\R^3}\rho\d_t\eta\,dxdt=-\int_{\R^3}\eta(0, \cdot)\rho_0\,dx+
		\int_0^T\langle\eta\psi, iF\rangle+\langle iF_\eta, \psi\rangle\,dt.
	\end{equation}
By using the definition of $F$ and $F_\eta$ given in \eqref{eq:def_VF} and \eqref{eq:F_eta}, respectively, we then have
	\begin{equation*}
		\begin{split}
			\langle\eta\psi, iF\rangle+\langle iF_\eta, \psi\rangle
			=&\langle\psi, 2i\eta F+i\nabla\eta\cdot\nabla\psi+\frac{i}{2}\psi\Delta\eta\rangle\\
			=&\langle\psi,-2\eta A\cdot\nabla\psi+i\nabla\eta\cdot\nabla\psi\rangle\\
			=&\int_{\R^3}-\eta A\cdot\nabla\rho+\nabla\eta\cdot\RE(\bar\psi(i\nabla)\psi)\,dx\\
			=&-\int_{\R^3}\nabla\eta\cdot J\,dx.
		\end{split}
	\end{equation*}
	By plugging the above identity into \eqref{eq:con_par} we then obtain the weak formulation \eqref{eq:cont} of the continuity equation.
\end{proof}


In what follows we provide a slight refinement of Proposition \ref{pr:continuity}, allowing for rougher test functions in \eqref{eq:cont}. This result will be used below when we prove the validity of the weak formulation for the momentum equation.

\begin{corollary}\label{pr:continuity_cor}
	Let us fix $\gamma\in(1,3)$, $T>0$, $\sigma>1$, and let $(\psi,A)$ be a weak $M^{1,1}$-solution to the system \eqref{eq:MS} on the space-time slab $[0,T]\times\R^3$, with initial data $(\psi_0,A_0,A_1)\in M^{1,\sigma}$. Then the integral identity
	\begin{equation}\label{eq:cont_cor}
		\int_0^T\int_{\R^3}\rho\partial_t\eta+J\cdot \nabla\eta\,dxdt+\int_{\R^3}\rho_0(x)\eta(0,x)dx=0
	\end{equation}
	holds true for every test function $\eta\in L_T^{\infty}H^{\sigma}\cap W^{1,\infty}_TL^2$, such that $\operatorname{supp}(\eta)$ is compact in $[0,T)\times\R^3$.
\end{corollary}

\begin{proof}
	Let us denote by $C(\eta)$ the l.h.s.~of \eqref{eq:cont_cor}. We are going to show  that the linear map $\eta\mapsto C(\eta)$ is continuous from $L_T^{\infty}H^{\sigma}\cap W^{1,\infty}_TL^2$ to $\R$. We start with the estimates
	\begin{gather}
		\label{mok_uno}\|\rho\partial_t\eta\|_{L_T^1L^1}\lesssim_T \|\rho\|_{L_T^{\infty}L^2} \|\partial_t\eta\|_{L_T^{\infty}L^2}\lesssim \|\eta\|_{W_T^{1,\infty}L^2},\\\label{mok_due}
		\|\rho_0\eta_0\|_{L^1}\lesssim\|\psi_0\|_{L^4}^2\|\eta_0\|_{L^2}\lesssim \|\eta\|_{W_T^{1,\infty}L^2}.
	\end{gather}
	For the term involving $J\cdot \nabla\eta$, we can use that $\psi\in L_T^2W^{\frac12,6}$ and Sobolev embedding to obtain
	\begin{equation}\label{mok_tre}
		\begin{split}
			\|J\cdot \nabla\eta\|_{L_T^1L^1}&\lesssim_T \|\psi\|_{L_T^2L^{3/(\sigma-1)}}\|\nabla\psi\|_{L_T^{\infty}L^2}\|\nabla\eta\|_{L_T^{\infty}L^{6/(5-2\sigma)}}\\
			&\lesssim\|\psi\|_{L_T^2W^{1/2,6}}\|\nabla\psi\|_{L_T^{\infty}L^2}\|\eta\|_{L_T^{\infty}H^{\sigma}}\lesssim_T \|\eta\|_{L_{T}^{\infty}H^{\sigma}},
		\end{split}
	\end{equation}
	where we assumed for simplicity $\sigma<2$. Combining estimates \eqref{mok_uno}, \eqref{mok_due} and \eqref{mok_tre} we then obtain the desired continuity. Observe moreover that, since $\operatorname{supp}(\eta)$ is compact in $[0,T)\times\R^3$, we can find a sequence of function $\eta^{(n)}\in \mathcal{C}^{\infty}_c([0,T)\times\R^3)$ such that $\eta^{(n)}\rightarrow\eta$ in $L_T^{\infty}H^{\sigma}\cap W^{1,\infty}_TL^2$. By Proposition \ref{pr:continuity} we have $C(\eta^{(n)})=0$, whence $C(\eta)=\lim C(\eta^{(n)})=0$, which proves the thesis.
\end{proof}

\subsection{Momentum equation}
We are going to prove the following result.
\begin{proposition}\label{pr:momentum}
	Let us fix $\gamma\in(1,3)$, $T>0$, $\sigma>1$, and let $(\psi,A)$ be a weak $M^{1,1}$-solution to the system \eqref{eq:MS} on the space-time slab $[0,T]\times\R^3$, with initial data $(\psi_0,A_0,A_1)\in M^{1,\sigma}$. Then the integral identity 
	\begin{equation}\label{eq:mome}
		\begin{split}
			\int_0^T\int_{\R^3}J&\cdot\partial_t\zeta+\nabla\zeta:\big(\nabla\sqrt{\rho}\otimes\nabla\sqrt{\rho}+\Lambda\otimes\Lambda\big)+(\rho E+J\wedge B)\cdot\zeta\\
			&-\frac{1}{4}\rho\Delta\diver\zeta+P(\rho)\diver\zeta dxdt+\int_{\R^3}J_0(x)\cdot\zeta(0,x)dx=0;
		\end{split}
	\end{equation}
	holds true for every $\zeta\in\mathcal{C}^{\infty}_c([0,T)\times\R^3;\R^3)$.
\end{proposition}

We will split the proof of Proposition \ref{pr:momentum} in a series of lemmas. Let us fix a test function $\zeta\in\mathcal{C}^{\infty}_c([0,T)\times\R^3;\R^3)$, and for any index $j=1,2,3$ we define $F_{\zeta^j}$ as in \eqref{eq:F_eta}. We start by rewriting the weak formulation of the momentum equation in terms of $(\psi, A)$, by also making some simplifications.

\begin{lemma}\label{le:si}
	Let us fix $\gamma\in(1,3)$, $T>0$, $\sigma>1$, and let $(\psi,A)$ be a weak $M^{1,1}$-solution to the system \eqref{eq:MS} on the space-time slab $[0,T]\times\R^3$, with initial data $(\psi_0,A_0,A_1)\in M^{1,\sigma}$. Then the following identity holds:
	\begin{equation}\label{eq:w_mom_hyb}
		\begin{split}
			\int_0^T\int_{\R^3}J&\partial_t\zeta+\nabla\zeta:\big(\nabla\sqrt{\rho}\otimes\nabla\sqrt{\rho}+\Lambda\otimes\Lambda\big)+(\rho E+J\wedge B)\cdot\zeta\\
			&-\frac{1}{4}\rho\Delta\diver\zeta+P(\rho)\diver\zeta dxdt+\int_{\R^3}J_0(x)\cdot\zeta(0,x)dx\\
			=\int_0^T\int_{\R^3}&\IM(\bar\psi\nabla\psi)\cdot\d_t\zeta+\nabla\zeta:\RE(\nabla\bar\psi\otimes\nabla\psi)\\
			&+\left(f'(\rho)+\phi+\frac{|A|^2}{2}\right)\diver(\rho\zeta)-\IM(\bar\psi\nabla\psi)\cdot A\diver\zeta\\
			&-\zeta\cdot\left(A\wedge\IM(\nabla\bar\psi\wedge\nabla\psi)\right)
			-\frac14\rho\Delta\diver\zeta\,dxdt\\
			&+\int_{\R^3}\IM(\bar\psi_0\nabla\psi_0)\cdot\zeta(0, \cdot)\,dx.
		\end{split}
	\end{equation}
\end{lemma}
\begin{proof}
	The proof is a straightforward consequence of the definitions \eqref{eq:mad}, \eqref{eq:em_f} and of the polar factorization technique for the wave function $\psi$ to define the hydrodynamic state $(\sqrt{\rho}, \Lambda)$ as in Lemma \ref{le:polar}.
	However we provide it for the sake of clarity in the exposition.
	
	Using the identities \eqref{eq:em_f} and \eqref{eq:id-pol}, the left hand side of \eqref{eq:w_mom_hyb} becomes
	\begin{equation}\label{eq:hyb1}\begin{aligned}
			\int_0^T\int_{\R^3}&J\cdot\d_t\zeta
			+\nabla\zeta:\RE\left(\overline{(-i\nabla-A)\psi}\otimes(-i\nabla-A)\psi\right)
			+P(\rho)\diver\zeta\\
			&+\rho(-\nabla\phi-\d_tA)\cdot\zeta
			+(J\wedge(\nabla\wedge A))\cdot\zeta
			-\frac14\rho\Delta\diver\zeta\,dxdt\\
			+\int_{\R^3}&J_0\cdot\zeta(0, \cdot)\,dx.
	\end{aligned}\end{equation}
 Let us set $f(\rho):=\gamma^{-1}\rho^{\gamma}$, and notice that $\nabla P(\rho)=\rho\nabla f'(\rho)$ in the sense of distributions. Moreover we can expand the bilinear term depending on the covariant derivative of $\psi$ as
	\begin{multline*}
		\nabla\zeta:\RE\left(\overline{(-i\nabla-A)\psi}\otimes(-i\nabla-A)\psi\right)\\
		=\d_j\zeta^k\RE(\d_j\bar\psi\d_k\psi)-\d_j\zeta^k\left(A^kJ^j+A^jJ^k+\rho A^jA^k\right)\\
		=\nabla\zeta:\RE(\nabla\bar\psi\otimes\nabla\psi)-J\cdot((A\cdot\nabla)\zeta)-A((J\cdot\nabla)\zeta)
		+\rho\nabla\zeta:(A\otimes A).
	\end{multline*}
	Furthermore, by using formula \eqref{eq:tril} we can express
	\begin{equation*}
		(J\wedge(\nabla\wedge A))\cdot\zeta=J\cdot((\zeta\cdot\nabla)A)-\zeta\cdot((J\cdot\nabla)A).
	\end{equation*}
	By using the previous identities, the expression \eqref{eq:hyb1} then becomes
	\begin{equation}\begin{aligned}\label{eq:hyb2}
			\int_0^T\int_{\R^3}&J\cdot\d_t\zeta+\nabla\zeta:\RE(\nabla\bar\psi\otimes\nabla\psi)-\rho\nabla\zeta:(A\otimes A)
			-J\cdot((A\cdot\nabla)\zeta)\\
			&-A\cdot((J\cdot\nabla)\zeta)
			+J\cdot((\zeta\cdot\nabla)A)-\zeta\cdot((J\cdot\nabla)A)\\
			&+(f'(\rho)+\phi)\diver(\rho\zeta)-\rho\zeta\cdot\d_tA
			-\frac14\rho\Delta\diver\zeta\,dxdt\\
			+\int_{\R^3}&J_0\cdot\zeta(0, \cdot)\,dx.
	\end{aligned}\end{equation}
	Let us notice that in the integral above we have the following term
	\begin{equation}\label{firte}
		-\int_0^T\int_{\R^3}A\cdot((J\cdot\nabla)\zeta)+\zeta\cdot((J\cdot\nabla)A)\,dxdt=-\int_0^T\int_{\R^3}(J\cdot\nabla)(A\cdot\zeta)\,dxdt.
	\end{equation}
	Let us also observe that, in view of the a priori bound \eqref{eq:apriori_uno}, we have $A\in L_T^{\infty}H^{\widetilde{\sigma}}$ for some $\widetilde{\sigma}\in(1,\frac76)$. By applying Corollary \ref{pr:continuity_cor} with $\eta=A\cdot\zeta\in L_T^{\infty}H^{\tilde{\sigma}}\cap W_T^{1,\infty}L^2$ we then obtain
	\begin{equation}\label{eq:conAz}
		-\int_0^T\int_{\R^3}(J\cdot\nabla)(A\cdot\zeta)\,dxdt=\int_0^T\int_{\R^3}\rho\partial_t(A\cdot\zeta)+\int_{\R^3}\rho_0(x)A_0\cdot\zeta(0,x)dx.
	\end{equation}
	In \eqref{eq:hyb2} we can also isolate the following terms
	\begin{equation}\label{alterm}
		\begin{split}
			-\int_0^T\int_{\R^3}&J\cdot((A\cdot\nabla)\zeta)-J\cdot((\zeta\cdot\nabla)A)\,dxdt\\
			&=\int_0^T\int_{\R^3}\zeta\cdot((A\cdot\nabla)J)-A\cdot((\zeta\cdot\nabla)J)-A\cdot J\diver\zeta\,dxdt\\
			&=-\int_0^T\int_{\R^3}(A\wedge(\nabla\wedge J))\cdot\zeta+A\cdot J\diver\zeta\,dxdt,
		\end{split}
	\end{equation}
	where in the last equality we used again \eqref{eq:tril}. Using \eqref{firte}, \eqref{eq:conAz} and \eqref{alterm}, and the identity $J=\IM(\bar\psi\nabla\psi)-\rho A$, the expression \eqref{eq:hyb2} can be written as
	\begin{equation}\label{eq:hyb3}\begin{aligned}
			\int_0^T\int_{\R^3}&\IM(\bar\psi\nabla\psi)\cdot\d_t\zeta+\nabla\zeta:\RE(\nabla\bar\psi\otimes\nabla\psi)
			-\rho\nabla\zeta:(A\otimes A)\\
			&-(A\wedge(\nabla\wedge J))\cdot\zeta-A\cdot J\diver\zeta
			+(f'(\rho)+\phi)\diver(\rho\zeta)\\
			&-\frac14\rho\Delta\diver\zeta\,dxdt+\int_{\R^3}\IM(\bar\psi_0\nabla\psi_0)\cdot\zeta(0, \cdot)\,dx.
	\end{aligned}\end{equation}
	Let us now consider the two remaining terms involving $J$. We have 
	\begin{gather}
		\label{eq:ga1}\nabla\wedge J=\nabla\wedge\IM(\bar\psi\nabla\psi)-\nabla\wedge(\rho A)
		=\IM(\nabla\bar\psi\wedge\nabla\psi)-\nabla\wedge(\rho A)\\\label{eq:ga2}
		-A\cdot J\diver\zeta=\rho|A|^2\diver\zeta-A\cdot\IM(\bar\psi\nabla\psi)\diver\zeta
	\end{gather}
	Moreover, by integrating by parts the term with $\rho\nabla\zeta:(A\otimes A)$, we obtain
	\begin{equation*}\begin{aligned}
			-\int_0^T\int_{\R^3}&\rho\nabla\zeta:(A\otimes A)+\zeta\cdot(A\wedge(\nabla\wedge J))+A\cdot J\diver\zeta\,dxdt\\
			=\int_0^T\int_{\R^3}&\zeta\cdot((A\cdot\nabla)(\rho A))+\zeta\cdot(A\wedge(\nabla\wedge(\rho A)))+\rho|A|^2\diver\zeta\\
			&-\zeta\cdot(A\wedge\IM(\nabla\bar\psi\wedge\nabla\psi))
			-A\cdot\IM(\bar\psi\nabla\psi)\diver\zeta\,dxdt.
	\end{aligned}\end{equation*}
	By \eqref{eq:tril} we know that the first two terms on the right hand side of the previous identity equal $A\cdot((\zeta\cdot\nabla)(\rho A))$, whence
	\begin{equation}\label{eq:ga3}
		\begin{aligned}
			-\int_0^T\int_{\R^3}&\rho\nabla\zeta:(A\otimes A)+\zeta\cdot(A\wedge(\nabla\wedge J))+A\cdot J\diver\zeta\,dxdt\\
			=\int_0^T\int_{\R^3}&A\cdot((\zeta\cdot\nabla)(\rho A))+\rho|A|^2\diver\zeta\\
			&-(A\wedge\IM(\nabla\bar\psi\wedge\nabla\psi))\cdot\zeta
			-A\cdot\IM(\bar\psi\nabla\psi)\diver\zeta\,dxdt\\
			=\int_0^T\int_{\R^3}&\frac{|A|^2}{2}\diver(\rho\zeta)
			-(A\wedge\IM(\nabla\bar\psi\wedge\nabla\psi))\cdot\zeta
			-A\cdot\IM(\bar\psi\nabla\psi)\diver\zeta\,dxdt,
	\end{aligned}\end{equation}
	where in the last equality we just integrated by parts once again. Combining \eqref{eq:ga1}, \eqref{eq:ga2} and \eqref{eq:ga3}, the expression \eqref{eq:hyb3} then becomes
	\begin{equation*}\begin{aligned}
			\int_0^T\int_{\R^3}&\IM(\bar\psi\nabla\psi)\cdot\d_t\zeta+\nabla\zeta:\RE(\nabla\bar\psi\otimes\nabla\psi)
			+\left(f'(\rho)+\phi+\frac{|A|^2}{2}\right)\diver(\rho\zeta)\\
			&-\IM(\bar\psi\nabla\psi)\cdot A\diver\zeta-\zeta\cdot\left(A\wedge\IM(\nabla\bar\psi\wedge\nabla\psi)\right)
			-\frac14\rho\Delta\diver\zeta\,dxdt\\
			+\int_{\R^3}&\IM(\bar\psi_0\nabla\psi_0)\cdot\zeta(0, \cdot)\,dx,
	\end{aligned}\end{equation*}
	that is the right hand side of \eqref{eq:w_mom_hyb}. The proof is complete.
\end{proof}
Next lemma provides somehow the analogue of formula \eqref{eq:con_par} for the momentum density. 
\begin{lemma}\label{prop:3}
	Let us fix $\gamma\in(1,3)$, $T>0$, $\sigma>1$, and let $(\psi,A)$ be a weak $M^{1,1}$-solution to the system \eqref{eq:MS} on the space-time slab $[0,T]\times\R^3$, with initial data $(\psi_0,A_0,A_1)\in M^{1,\sigma}$.Then the following identity holds:
	\begin{equation}\label{eq:ide_rip}
		\begin{split}
			\int_0^T\langle\psi\d_t\zeta^j,-i\d_j\psi\rangle
			+\langle \zeta^jF+F_{\zeta^j}, \d_j\psi\rangle
			+&\langle\psi\d_j\zeta^j, F\rangle \,dt\\
			+&\langle\psi_0\zeta^j(0), -i\d_j\psi_0\rangle=0.
		\end{split}
	\end{equation}
\end{lemma}
\begin{proof}
	The proof is similar to that of Proposition \ref{pr:continuity}. For this reason here we might skip some passages which are completely analogous to the analysis performed for the Proposition \ref{pr:continuity}, in order to highlight the main differences, especially in the rigorous justification of some passages related to the integrability properties of $(\psi, A)$. 
	\newline
	As before, by using the definition of $J$ as in \eqref{eq:mad}, we consider
	\begin{equation}\label{eq:kng}
		\begin{split}
			\int_0^T\langle\psi\d_t\zeta^j, -i\d_j\psi&\rangle\,dt=\int_0^T\big\langle\psi\d_t\zeta^j, -ie^{\frac{i}{2}t\Delta}\d_j\psi_0\\
			&-\d_j\int_0^te^{\frac{i}{2}(t-s)\Delta}F(s)\,ds\big\rangle\,dt:=I+II,
		\end{split}
	\end{equation}
	where we also used the Duhamel's formula \eqref{eq:duh}. Now we treat the two contributions separately.
	
	\textbf{Term I.} 
	Using Fubini Theorem and the identity \eqref{duha_eta} we obtain
	\begin{equation*}\begin{aligned}
			I=&\Big\langle\int_0^Te^{\frac{i}{2}(T-t)\Delta}(\psi\d_t\zeta^j)(t)\,dt,
			-ie^{\frac{i}{2}T\Delta}\d_j\psi_0\Big\rangle\\
			=&\Big\langle(\psi\zeta^j)(T)-e^{\frac{i}{2}T\Delta}(\psi\zeta^j)(0)
			+i\int_0^Te^{\frac{i}{2}(T-t)\Delta}F_{\zeta^j}(t)\,dt, -ie^{\frac{i}{2}T\Delta}\d_j\psi_0\Big\rangle.
	\end{aligned}\end{equation*}
	The unitarity of the propagator $e^{\frac{i}{2}t\Delta}$ and the relation $\zeta^j(T)\equiv 0$ then yield
	\begin{equation}\label{est_muno}
		I=-\big\langle(\psi\zeta^j)(0), -i\d_j\psi_0\big\rangle
		+\int_0^T\big\langle iF_{\zeta^j}(t), -ie^{\frac{i}{2}t\Delta}\d_j\psi_0\big\rangle\,dt.
	\end{equation}
	
	\textbf{Term II.}
	Using again identity \eqref{duha_eta}, the unitarity of the propagator $e^{\frac{i}{2}t\Delta}$, and the relation $\zeta^j(T)\equiv 0$ we obtain
	\begin{equation}\label{est_mdue}
		\begin{aligned}
			II=&\int_0^T\Big\langle\int_s^Te^{\frac{i}{2}(s-t)\Delta}(\psi\d_t\zeta^j)(t)\,dt, -\d_jF(s)\Big\rangle\,ds\\
			=&\int_0^T\Big\langle-(\psi\zeta^j)(s)-i\int_T^se^{\frac{i}{2}(s-t)\Delta}F_{\zeta^j}(t)\,dt, -\d_jF(s)\Big\rangle\,ds,
		\end{aligned}
	\end{equation}
	where the various exchanges of order of integration are justified by estimate \eqref{fub}.

In order to proceed with the rigorous manipulation of term $II$, we need more precise estimates on $F$. To this aim, observe that
\begin{equation*}
	\begin{split}
\|F\|_{L_T^2(L^2+W^{1/2,6/5})}&\lesssim \|iA\nabla\psi+({\textstyle{\frac12}}|A|^2+\phi)\psi\|_{L_T^2L^2}+\||\psi|^{2(\gamma-1)}\psi\|_{L_T^2W^{1/2,6/5}}\\
&\lesssim_T \|A\|_{L_T^2L^{\infty}}\|\nabla\psi\|_{L_T^{\infty}L^2}+\|A\|_{L_T^{\infty}L^6}^2\|\psi\|_{L_T^{\infty}L^6}+\|\psi\|_{L_T^{\infty}H^1}^3\\
&\quad +\|\psi\|^{2(\gamma-1)}_{L_T^{\infty}L^{6(\gamma-1)\wedge 6}}\|\psi\|_{L_T^2W^{1/2+, 6/(7-2\gamma)\vee 2}},
\end{split}
\end{equation*}
where we used estimate \eqref{esti_conv} and Lemma \ref{esti_pure}. Combining the estimate above with the a priori bound \eqref{eq:apriori_uno} and the embedding $H^1\cap W^{1/2,6}\hookrightarrow W^{1/2+,6/(7-2\gamma)}$ we deduce
\begin{equation}\label{refinedF}
\|F\|_{L_T^2(L^2+W^{1/2,6/5})}\lesssim_T 1.
\end{equation}
As a direct consequence of \eqref{refinedF} and the definition \eqref{eq:F_eta} we also obtain
\begin{equation}\label{refinedFeta}
	\|F_{\zeta^j}\|_{L_T^2(L^2+W^{1/2,6/5})}\lesssim_T 1.	
\end{equation}
Observe moreover that the Duhamel formulas \eqref{eq:duh} and \eqref{duha_eta} imply
	\begin{gather}
		\label{impr_i}i\int_0^te^{\frac{i}{2}(t-s)\Delta}F(s)ds=e^{\frac{i}{2}t\Delta}\psi_0-\psi(t)\in L_T^{2}(H^1\cap W^{1/2,6}),\\
		\label{impr_ii}\begin{split}i\int_T^se^{\frac{i}{2}(s-t)\Delta}&F_{\zeta^j}(t)\,dt=-(\zeta^j\psi)(s)\\
			&+\int_T^se^{\frac{i}{2}(s-t)\Delta}(\psi\partial_t\zeta^j)(t)dt\in L_T^{2}(H^1\cap W^{1/2,6}),
		\end{split}
	\end{gather}
where we used the a priori bound \eqref{eq:apriori_uno} and the Strichartz estimates \eqref{clasti}-\eqref{class_stri}. For $\eps>0$, let us also consider the Yosida approximation operator $\mathcal{Y}_{\eps}:=(1-\eps\Delta)^{-1}$, and recall that $\mathcal{Y}_{\eps}f\rightarrow f$ in $W^{s,p}(\R^3)$ as $\eps\to 0$, for every $s\in\R$, $p\in(1,\infty)$ and $f\in W^{s,p}(\R^3)$. We have the chain of identities
	\begin{equation}\label{est_mtre}
		\begin{aligned}
			&\int_0^T\Big\langle-i\int_T^se^{\frac{i}{2}(s-t)\Delta}F_{\zeta^j}(t)\,dt, -\d_jF(s)\Big\rangle\,ds\\
			&=\lim_{\eps\to 0}\int_0^T\Big\langle-i\int_T^se^{\frac{i}{2}(s-t)\Delta}\mathcal{Y}_{\eps}F_{\zeta^j}(t)\,dt, -\d_jF(s)\Big\rangle\,ds\\
			&=\lim_{\eps\to 0}\int_0^T\big\langle i\mathcal{Y}_{\eps}F_{\zeta^{j}}(t),-\partial_j\int_0^te^{\frac{i}{2}(t-s)\Delta}F(s)ds\big\rangle\,dt\\
			&=\int_0^T\big\langle iF_{\zeta^{j}}(t),-\partial_j\int_0^te^{\frac{i}{2}(t-s)\Delta}F(s)ds\big\rangle\,dt,
		\end{aligned}
	\end{equation}
where the first and third steps are justified by estimates \eqref{refinedF}-\eqref{impr_ii}, which guarantee that all the duality products are well-defined, while in the second step we used the unitarity of the propagator $e^{\frac{i}{2}t\Delta}$ and Fubini Theorem, the exchange of order of integration being justified by estimates \eqref{eq:est_V} and \eqref{fubeta}.
	
	Combining \eqref{est_mdue} and \eqref{est_mtre} we deduce
	\begin{equation}\label{est_mquattro}
		II=\int_0^T\langle(\psi\zeta^{j})(s),\partial_j F(s)\rangle\,ds+\int_0^T\big\langle iF_{\zeta^{j}}(t),-\partial_j\int_0^te^{\frac{i}{2}(t-s)\Delta}F(s)ds\big\rangle\,dt.
	\end{equation}

	\textbf{I+II.} Combining the identities \eqref{eq:kng}, \eqref{est_muno} and \eqref{est_mquattro}, and using the Duhamel formula \eqref{eq:duh}  we obtain
	\begin{equation*}\begin{aligned}
			\int_0^T\langle&\psi\d_t\zeta^j, -i\d_j\psi\rangle\,dt=-\langle(\psi\zeta^j)(0), -i\d_j\psi_0\rangle\\
			&\quad+\int_0^T\langle iF_{\zeta^j}(t), -i\d_j\psi(t)\rangle
			+\langle\psi\zeta^j, \d_jF\rangle\,dt\\
			&=-\langle \psi_0\zeta^j(0), -i\d_j\psi_0\rangle-\int_0^T\langle \zeta^jF+F_{\zeta^j}, \d_j\psi\rangle
			+\langle\psi\d_j\zeta^j, F\rangle \,dt,
	\end{aligned}\end{equation*}
	where in the last step we integrate by parts the term $\langle\psi\zeta^j, \d_jF\rangle$. The proof is complete.
\end{proof}
We are now able to prove Proposition \ref{pr:momentum}
\begin{proof}[Proof of Proposition \ref{pr:momentum}]
	Owing to Lemma \ref{le:si}, and rewriting the right hand side of identity \eqref{eq:w_mom_hyb} by using repeated indices, we deduce that it is sufficient to show
	\begin{equation}\label{eq:w_mom_hyb2}\begin{aligned}
			\int_0^T\int_{\R^3}&\IM(\bar\psi\d_j\psi)\d_t\zeta^j
			+\d_j\zeta^k\RE(\overline{\d_j\psi}\d_k\psi)
			+\left(f'(\rho)+\phi+\frac{|A|^2}{2}\right)\d_j(\rho\zeta^j)\\
			&-\IM(\bar\psi\d_k\psi) A^k\d_j\zeta^j
			-2\zeta^jA^k\IM(\overline{\d_j\psi}\d_k\psi)-\frac14\rho\d_{kkj}\zeta^j\,dxdt\\
			+\int_{\R^3}&\IM(\bar\psi_0\d_j\psi_0)\zeta^j(0, \cdot)\,dx=0.
	\end{aligned}\end{equation}
	To this aim, we use the following identities:
	\begin{equation*}
		\begin{aligned}
			\RE(\bar F\d_j\psi)=&\RE\left(-iA^k\overline{\d_k\psi}\d_j\psi\right)
			+\left(\frac{|A|^2}{2}+\phi+\rho^{\gamma-1}\right)\d_j\frac{\rho}{2},\\
			\RE(\bar F\psi)=&\RE(-iA^k\overline{\d_k\psi}\psi)
			+\left(\frac{|A|^2}{2}+\phi+\rho^{\gamma-1}\right)\rho.
		\end{aligned}
	\end{equation*}
	Combining the identities above with \eqref{eq:ide_rip} we obtain
	\begin{equation*}\begin{aligned}
			\int_0^T\int_{\R^3}&\IM(\bar\psi\d_j\psi)\d_t\zeta^j
			+\d_j\zeta^k\RE(\overline{\d_j\psi}\d_k\psi)
			+\left(\frac{|A|^2}{2}+\phi+\rho^{\gamma-1}\right)\d_j(\rho\zeta^j)\\
			&-\IM(\bar\psi\d_k\psi) A^k\d_j\zeta^j
			-2\zeta^jA^k\IM(\overline{\d_j\psi}\d_k\psi)-\frac14\rho\d_{kkj}\zeta^j\,dxdt\\
			+\int_{\R^3}&\IM(\bar\psi_0\d_j\psi_0)\zeta^j(0, \cdot)\,dx\\
			=\int_0^T\int_{\R^3}&\IM(\bar\psi\d_j\psi)\d_t\zeta^j
			-2\zeta^j\IM(\d_j\bar\psi\d_k\psi)
			+\left(\frac{|A|^2}{2}+\phi+\rho^{\gamma-1}\right)\zeta^j\d_j\rho\\
			&+\d_k\zeta^j\RE(\d_k\bar\psi\d_j\psi)+\frac14\Delta\zeta^j\d_j\rho
			-A^k\IM(\bar\psi\d_k\psi)\d_j\zeta^j\\
			&+\left(\frac{|A|^2}{2}+\phi+\rho^{\gamma-1}\right)\rho\d_j\zeta^j\,dxdt
			+\int_{\R^3}\IM(\bar\psi_0\d_j\psi_0)\zeta^j(0, \cdot)\,dx\\
			=\int_0^T\langle\psi&\d_t\zeta^j,-i\d_j\psi\rangle
			+\langle \zeta^jF+F_{\zeta^j}, \d_j\psi\rangle
			+\langle\psi\d_j\zeta^j, F\rangle \,dt+\langle\psi_0\zeta^j(0), -i\d_j\psi_0\rangle=0,
	\end{aligned}\end{equation*}
	which concludes the proof.
\end{proof}

\section{Local smoothing estimates}\label{sec:ls}
The aim of this section is to provide suitable local smoothing estimates for the nonlinear Maxwell-Schr\"odinger system. This smoothing effect will be crucial in the proof of Theorem \ref{th:stability}, as it provides the compactness needed to deduce the stability of the hydrodynamic variables and the Lorentz force. The main result of this section is the following.

\begin{proposition}\label{pr:smoot_MS}
Let us fix $\gamma\in(1,3)$, $\sigma\in(1,\frac76)$, and $T>0$. Let $(\psi,A)$ be a weak $M^{1,1}$-solution to the system \eqref{eq:MS} on the space-time slab $[0,T]\times\R^3$, with initial data $(\psi_0,A_0,A_1)\in M^{1,\sigma}$. Then for every $\delta\in(0,\sigma-1\wedge\frac{3-\gamma}{2})$ we have the estimate
\begin{equation}\label{eq:losmoMS}
	\|\chi_{\alpha}\psi\|_{\ell_{\alpha}^{\infty}L_T^2H^{1+\delta}}\lesssim_T \langle\|(\psi,A)\|_{L_T^{\infty}(H^1\times H^1)}\rangle^n\langle\|(A_0,A_1)\|_{\Sigma^{\sigma}}\rangle^n.
\end{equation}	
\end{proposition}

In order to prove Proposition \ref{pr:smoot_MS} we first obtain a smoothing effect for the linear Schr\"odinger equation with a fixed, time-dependent magnetic potential, and then we apply it to the case of the Maxwell-Schr\"odinger system.

\subsection{Smoothing effect for a fixed magnetic potential}
We assume here that the magnetic potential $A$ is fixed and satisfies the following assumption:

\begin{hyp}\label{ass:A}
$A\in L_T^{\infty}H^1\cap W_T^{1,\infty}L^2$ for some $T>0$, and $\diver A (t,\cdot)=0$ for every $t\in(0,T)$. Moreover, for every wave-admissible pair $(q,r)$ we have that $\chi_{\alpha}A\in \ell_{\alpha}^2L_T^{q}W^{1-\frac2q,r}$. 
\end{hyp}
We are going to show that if $A$ satisfies Assumption $\ref{ass:A}$ and $\operatorname{rot}A$ is slightly more regular than $L^2$, then the linear magnetic Schr\"odinger flow exhibits a local-smoothing effect.

\begin{proposition}\label{pr:local_smoothing}
Let us fix $s\in[0,2)$, $\sigma\in(1,\frac54)$ and $\delta\in(0,\min\{\sigma-1,2-s\})$. Suppose that $A$ satisfy Assumption \ref{ass:A}, and assume in addition
\begin{equation}\label{extra_rot}
\chi_{\alpha}\operatorname{rot}A\in\ell_{\alpha}^2L_T^{\infty}H^{\sigma-1}.
\end{equation}
Fix $F\in L_T^2H^{s+2\delta-1}$, and  let $\psi\in L_T^{\infty}H^s$ be a weak solution to the magnetic Schr\"odinger equation $i\partial_t\psi=-\Delta_A\psi+F$. Consider moreover the wave-admissible pairs $(q_1,r_1)$, $(q_2,r_2)$ given by $\frac{2}{r_1}=3-2\sigma+\delta$ and $\frac{1}{r_2}=2(\sigma-1-\delta)$. Then we have the estimate
\begin{equation}\label{eq:local_smoothing}
\|\chi_{\alpha}\psi\|_{\ell_{\alpha}^{\infty}L_T^2H^{s+\delta}}\lesssim_{T}\langle\|A\|_{\mathcal{X}_T}\rangle^n\Big(\|\psi\|_{L_T^{\infty}H^s}+\|F\|_{L_T^2H^{s+2\delta-1}}\Big),
\end{equation}
where we set
\begin{equation}\label{def:xt}
\begin{split}
\|A\|_{\mathcal{X}_T}&:=\|A\|_{L_T^{\infty}H^1\cap W_T^{1,\infty}L^2}+\|\chi_{\alpha}A\|_{\ell_{\alpha}^2L_T^{q_1}W^{1-\frac{2}{q_1},{r_1}}}\\
&\quad+\|A\|_{L_T^{q_2}W^{1-\frac{2}{q_2},{r_2}}}+\|\chi_{\alpha}\operatorname{rot}A\|_{\ell_{\alpha}^2L_T^{\infty}H^{\sigma-1}}.
\end{split}
\end{equation}
\end{proposition}

\begin{remark}
It is worth comparing Theorem \ref{pr:local_smoothing} with the local smoothing result for the magnetic Schr\"odinger equation in two dimensions, proved in \cite[Corollary 3.1]{Wada2012}. First of all, in the 2D case no extra-regularity assumption on the magnetic potential as \eqref{extra_rot} is required. Moreover, in 2D the gain of regularity for the Schr\"odinger flow is almost sharp (i.e.~a gain of almost $1/2$-derivatives), whilst Theorem \ref{pr:local_smoothing} provides a gain of at most $1/4$-derivatives, even in the case of a smooth magnetic potential.
\end{remark}

In order to prove Theorem \ref{pr:local_smoothing}, we preliminary need a refined version of the Koch-Tzvetkov estimates for the magnetic Schr\"odinger evolution, in the same spirit as in \cite[Lemma 3.1]{Wada2012}.

\begin{lemma}\label{le:improved-kt}
Suppose that $A$ satisfies Assumption \ref{ass:A}, and let us fix $s\in[0,2)$, $\theta\in(0,1)$, and $m\in\left(\frac{\theta-1}{2},\frac{2\theta-1}{2}\right)$. Set moreover $\widetilde{s}:=2m-\theta+1\in (0,\theta)$, and consider the wave-admissible pairs $(q_1,r_1)$, $(q_2,r_2)$ given by $\frac{2}{r_1}=\widetilde{s}$ and $\frac{1}{r_2}=\theta-\widetilde{s}$. Then we have the estimate
\begin{equation}\label{eq:improved-kt}
\|\chi_{\alpha}\mathcal{D}_{A}^{s-\theta}\psi\|_{\ell_{\alpha}^2L_T^2W^{m,6}}\lesssim_T \langle C_A\rangle^n\Big(\|\psi\|_{L_T^{\infty}H^s}+\|F\|_{L_T^2H^{s-2\theta+2m}}\Big),
\end{equation}
where $$C_A:=\|A\|_{L_T^{\infty}H^1\cap W_T^{1,\infty}L^2}+\|\chi_{\alpha}A\|_{\ell_{\alpha}^2L_T^{q_1}W^{1-\frac{2}{q_1},{r_1}}}+\|A\|_{L_T^{q_2}W^{1-\frac{2}{q_2},{r_2}}}.$$
\end{lemma}

\begin{remark}
In terms of gain of integrability, Lemma \ref{le:improved-kt} does not improve on the standard version of the magnetic Koch-Tzvetkov estimates provided by Lemma \ref{le:quattordici}. Here the improvement is in terms of summability with respect to the spatial localization.
\end{remark}

\begin{proof}
Observe that the function $\chi_{\alpha}\mathcal{D}_{A}^{s-\theta}\psi$ satisfies the equation
$$(i\partial_t+\Delta)(\chi_{\alpha}\mathcal{D}_{A}^{s-\theta}\psi)=g_{\alpha}+\chi_{\alpha}\mathcal{D}_A^{s-\theta}F,$$
where we set
$$g_{\alpha}:=2iA\nabla(\chi_{\alpha}\mathcal{D}_{A}^{s-\theta}\psi)+\chi_{\alpha}|A|^2\mathcal{D}_{A}^{s-\theta}\psi+i\chi_{\alpha}[\partial_t,\mathcal{D}_A^{s-\theta}]\psi-[\chi_{\alpha},\Delta_A]\mathcal{D}_A^{s-\theta}\psi.$$
We then obtain
\begin{equation}\label{eq:zeroth_final}
\begin{split}
&\|\chi_{\alpha}\mathcal{D}_{A}^{s-\theta}\psi\|_{\ell_{\alpha}^2L_T^2W^{m,6}}^2\\
&\lesssim T^{-1}\|\chi_{\alpha}\mathcal{D}_{A}^{s-\theta}\psi\|_{\ell_{\alpha}^2L_T^{2}H^{\theta}}^2+T\|g_{\alpha}+\chi_{\alpha}\mathcal{D}_A^{s-\theta}F\|_{\ell_{\alpha}^2L_T^2H^{2m-\theta}}^2\\
&\lesssim \|\mathcal{D}_{A}^{s-\theta}\psi\|_{L_T^{2}H^{\theta}}^2+T\|g_{\alpha}\|_{\ell_{\alpha}^2L_T^2H^{2m-\theta}}^2+T\|\mathcal{D}_A^{s-\theta}F\|_{L_T^2H^{2m-\theta}}^2\\
&\lesssim_T \langle\|A\|_{L_T^{\infty}H^1}^n\rangle\Big(\|\psi\|_{L_T^{\infty}H^s}^2+\|F\|_{L_T^2H^{s+2m-2\theta}}^2\Big)+\|g_{\alpha}\|_{\ell_{\alpha}^2L_T^2H^{2m-\theta}}^2,
\end{split}
\end{equation}
where we used estimate \eqref{eq:kt} in the first step, the bound \eqref{eq:disin} in the second step and the equivalence of norms \eqref{equicono} in the last step. Let us estimates the various terms appearing in $g_{\alpha}$. Using the fractional Leibniz rule, H\"older inequality in $\alpha,t$, the bound \eqref{eq:disin} and Sobolev embedding we obtain
\begin{equation}\label{gia_uno}
\begin{split}
\|A\nabla(\chi_{\alpha}\mathcal{D}_{A}^{s-\theta}\psi)\|_{\ell_{\alpha}^2L_T^2H^{2m-\theta}}&\lesssim \|\chi_{\alpha}A\mathcal{D}_{A}^{s-\theta}\psi\|_{\ell_{\alpha}^2L_T^2H^{2m-\theta+1}}\\
&\lesssim \|A\|_{L_T^2W^{1-2/q_2,r_2}}\|\mathcal{D}_{A}^{s-\theta}\psi\|_{L_T^{\infty}H^{\theta}}\\
&\quad+\|\chi_{\alpha}A\|_{\ell_{\alpha}^2L_T^{q_1}W^{1-2/q_1,r_1}}\|\mathcal{D}_{A}^{s-\theta}\psi\|_{L_T^{2/\widetilde{s}}L^{2/(1-\widetilde{s})}}.
\end{split}
\end{equation}
Observe that, owing to the hypothesis on $m$ and $\theta$, we can find $\delta\in(0,1)$ such that
\begin{equation}\label{gia_tre}
\|\mathcal{D}_{A}^{s-\theta}\psi\|_{L_T^{2/\widetilde{s}}L^{2/(1-\widetilde{s})}}\lesssim_T \|\mathcal{D}_{A}^{s-\theta}\psi\|_{L_T^2W^{m,6}}^{1-\delta}\|\mathcal{D}_{A}^{s-\theta}\psi\|_{L_T^{\infty}H^{\theta}}^{\delta}.
\end{equation}
Combining \eqref{gia_uno}, \eqref{gia_tre}, and using the Young inequality and the equivalence of norms \eqref{equicono} we get that for every $\varepsilon>0$
 \begin{equation}\label{eq:first_final}
\|A\nabla(\chi_{\alpha}\mathcal{D}_{A}^{s-\theta}\psi)\|_{\ell_{\alpha}^2L_T^2H^{2m-\theta}}\lesssim \langle C_A\rangle^n\Big(\varepsilon\|\mathcal{D}_{A}^{s-\theta}\psi\|_{L_T^2W^{m,6}}+\varepsilon^{1-\frac{1}{\delta}}\|\psi\|_{L_T^{\infty}H^{s}}\Big).
\end{equation}

For the second term in $g_{\alpha}$, we have the bound
\begin{equation}\label{eq:second_final}
\begin{split}
\|\chi_{\alpha}|A|^2\mathcal{D}_{A}^{s-\theta}\psi\|_{\ell_{\alpha}^2L_T^2H^{2m-\theta}}&\lesssim \||A|^2\mathcal{D}_{A}^{s-\theta}\psi\|_{L^{6/(5+2\widetilde{s})}}\\
&\lesssim_T \|A\|_{L_T^{\infty}L^{3/(m+1)}}^2\|\mathcal{D}_{A}^{s-\theta}\psi\|_{L^{6/(3-2\theta)}}\\
&\lesssim \|A\|_{L_T^{\infty}H^1}^2\|\mathcal{D}_{A}^{s-\theta}\psi\|_{H^{\theta}}\lesssim  \langle C_A\rangle^n\|\psi\|_{L_T^{\infty}H^s},
\end{split}
\end{equation}
where we used \eqref{eq:disin}, Sobolev embedding and the equivalence of norms \eqref{equicono}.

For the third term, we first observe that
\begin{equation}\label{eq:third_uno}
\|\chi_{\alpha}[\partial_t,\mathcal{D}_A^{s-\theta}]\psi\|_{\ell_{\alpha}^2L_T^2H^{2m-\theta}}\lesssim_T \|[\partial_t,\mathcal{D}_A^{s-\theta}]\psi\|_{L_T^{\infty}H^{2m-\theta}},
\end{equation}
as it follows from the bound \eqref{eq:disin}. Let us prove now the following estimate:
\begin{equation}\label{eq:third_due}
\|[\partial_t,\mathcal{D}_A^{s-\theta}]\psi\|_{L_T^{\infty}H^{2m-\theta}}\lesssim \langle C_A\rangle^n \|\psi\|_{L_T^{\infty}H^{s}}.
\end{equation}
When $s-\theta\leq 0$, \eqref{eq:third_due} can be deduced by applying Lemma \ref{lela} with $s_1=2m-\theta$, $s_2=s$ and $s_3=\theta-s$. When $s-\theta>0$, the bound \eqref{eq:third_due} is instead obtained as follows:
\begin{equation}
\begin{split}
\|[\partial_t,\mathcal{D}_A^{s-\theta}]\psi\|_{L_T^{\infty}H^{2m-\theta}}&= \|\mathcal{D}_A^{s-\theta}[\mathcal{D}_A^{\theta-s},\partial_t]\mathcal{D}_A^{s-\theta}\psi\|_{L_T^{\infty}H^{2m-\theta}}\\
&\lesssim \langle C_A\rangle^n\|[\mathcal{D}_A^{\theta-s},\partial_t]\mathcal{D}_A^{s-\theta}\psi\|_{L_T^{\infty}H^{s+2m-2\theta}}\\
&\lesssim \langle C_A\rangle^n\|\mathcal{D}_A^{s-\theta}\psi\|_{L_T^{\infty}H^{\theta}}\lesssim \langle C_A\rangle^n\|\psi\|_{L_T^{\infty}H^s},
\end{split}
\end{equation}
where we used the identity $[X,Y]=X[Y,X^{-1}]X$ in the first step, the equivalence of norms \eqref{equicono} in the second and last steps, and Lemma \ref{lela} with $s_1=s+2m-2\theta$, $s_2=\theta$ and $s_3=s-\theta$ in the third step. Combining \eqref{eq:third_uno} and \eqref{eq:third_due} we deduce
\begin{equation}\label{eq:third_final}
\|[\partial_t,\mathcal{D}_A^{s-\theta}]\psi\|_{L_T^2H^{2m-\theta}}\lesssim_T  \langle C_A\rangle^n\|\psi\|_{L_T^{\infty}H^{s}}.
\end{equation}

Finally, let us consider the last term in $g_{\alpha}$. We have
\begin{equation}\label{eq:fourth_final}
\begin{split}
\|[\chi_{\alpha},\Delta_A]&\mathcal{D}^{s-\theta}\psi\|_{\ell_{\alpha}^2L_T^2H^{2m-\theta}}\lesssim_T \|\mathcal{D}^{s-\theta}\psi\|_{L_T^{\infty}H^{2m-\theta+1}}\\
&\lesssim \langle C_A\rangle^n\|\psi\|_{L_T^{\infty}H^{s+2m-2\theta+1}}\lesssim \langle C_A\rangle^n\|\psi\|_{L_T^{\infty}H^s},
\end{split}
\end{equation}
where we used \eqref{eq:disin}, the identity $[\chi_{\alpha},\Delta_A]=-2i\nabla\chi_{\alpha}\cdot\nabla_A+\Delta\chi_{\alpha}$, the fact $\operatorname{supp}\{\chi_{\alpha}\}$ overlap finitely, and the equivalence of norms \eqref{equicono}.

Combining \eqref{eq:zeroth_final} together with the estimates \eqref{eq:first_final}, \eqref{eq:second_final}, \eqref{eq:third_final} and \eqref{eq:fourth_final}, and choosing $\varepsilon$ sufficiently small in \eqref{eq:first_final}, we eventually deduce the Koch-Tzvetkov bound \eqref{eq:improved-kt}. The proof is complete.
\end{proof}

We will also need an estimate for the commutator $[\partial_t,\mathcal{D}_A^s]$, which is analogous to the result in \cite[Lemma 2.6]{Wada2012}.

\begin{lemma}\label{lela}
Let us fix $A\in L_T^{\infty}H^1\cap W_T^{1,\infty}L^2$. Fix also $s_1\in(-1,2)$, $s_2\in[0,\frac52)$, and $s_3\in[0,2)$, with $s_1-s_2\in(-\frac52,\frac32)$ and $s_3+s_2-s_1>\frac12$. Then the following estimate holds true:
\begin{equation}\label{eq:combou}
\|[\mathcal{D}_A^{-s_3},\partial_t]f\|_{L_T^{\infty}H^{s_1}}\lesssim\|A\|^n_{L_T^{\infty}H^1\cap W_T^{1,\infty}L^2}\|f\|_{L_T^{\infty}H^{s_2}}.
\end{equation}
\end{lemma}

\begin{proof}
By functional calculus we can write 
$$\mathcal{D}_A^{-s_3}=\pi^{-1}\sin{(\pi s_3/2)}\int_0^{+\infty}\lambda^{-s_3/2}R(\lambda)d\lambda.$$
Using the relation $[\partial_t,-\Delta_A]=2i\partial_t A\cdot\nabla_A$ we then obtain
\begin{equation}\label{eq:commutator_exp}
[\mathcal{D}_A^{-s_3},\partial_t]f=i\pi^{-1}\sin{(\pi s_3/2)}\int_{0}^{+\infty}\lambda^{-s_3/2}R_A(\lambda)\big(\partial_tA\cdot \nabla_AR_A(\lambda)f\big)d\lambda.
\end{equation}

Owing to the assumptions on $s_1$ and $s_2$, we can find $\widetilde{s}\in(-1,0)$ such that $s_1\in[\widetilde{s},\widetilde{s}+2]$ and $\frac52+\widetilde{s}\in[s_2,s_2+2]$. Hence Lemma \eqref{lessi} and the equivalence of norms \eqref{equicono} yield the fixed-time estimate
\begin{equation}\label{pfiz}
\begin{split}
\|R_A(\lambda)\big(\partial_tA&\cdot \nabla_AR_A(\lambda)f\big)\|_{H^{s_1}}\lesssim \langle\lambda\rangle^{-1+\frac{s_1-\widetilde{s}}{2}}\|A\|_{H^1}^n\|\partial_tA\cdot \nabla_AR_A(\lambda)f\|_{H^{\widetilde{s}}}\\
&\lesssim\langle\lambda\rangle^{-1+\frac{s_1-\widetilde{s}}{2}}\|A\|_{H^1}^n\|\partial_tA\|_{L^2}\|\nabla_AR_A(\lambda)f\|_{H^{3/2+\widetilde{s}}}\\
&\lesssim\langle\lambda\rangle^{-1+\frac{s_1-\widetilde{s}}{2}}\|A\|_{H^1}^n\|\partial_tA\|_{L^2}\|R_A(\lambda)f\|_{H_A^{5/2+\widetilde{s}}}\\
&\lesssim \langle\lambda\rangle^{-\frac34+\frac{s_1-s_2}{2}}\|A\|_{H^1}^n\|\partial_tA\|_{L^2}\|f\|_{H^{s_2}}.
\end{split}
\end{equation}
Combining \eqref{eq:commutator_exp} and \eqref{pfiz} we obtain
$$\|[\mathcal{D}_A^{-s_3},\partial_t]f\|_{H^{s_1}}\lesssim\|A\|_{H^1}^n\|\partial_tA\|_{L^2}\|f\|_{H^{s_2}}\int_0^{\infty}\langle\lambda\rangle^{-\frac34+\frac{s_1-s_2-s_3}{2}}d\lambda.$$
By hypothesis $s_1-s_2-s_3<-1/2$, whence the integral in the r.h.s.~of the estimate above is finite. Then we get the (fixed-time) bound
$$\|[\mathcal{D}_A^{-s_3},\partial_t]f\|_{H^{s_1}}\lesssim\|A\|_{H^1}^n\|\partial_tA\|_{L^2}\|f\|_{H^{s_2}},$$
which in particular implies \eqref{eq:combou}.
\end{proof}

We can prove now our main result.

\begin{proof}[Proof of Theorem \ref{pr:local_smoothing}]
We start by considering the case $s=0$. Let $h$ be a smooth, real valued, increasing function, such $h'(t)=1$ for $|t|\geq\frac12$, $h'(t)=0$ for $|t|\geq 1$, and $\sup_{t\in\R}|h(t)|\leq 1$. For every $\alpha\in\Z^3$ and spatial direction $j\in\{1,2,3\}$, we set $h_{\alpha,j}(x):=h
(x_j-\alpha_j)$, and with a slight abuse of notation we write $h_{\alpha,j}':=\partial_jh_{\alpha,j}$, and analogously for higher derivatives. Let us set moreover $\theta:=1-\delta\in (\frac34,1)$, and define the operator $L_{\alpha,j}:=\mathcal{D}_A^{-\theta}h_{\alpha,j}\partial_{A_j}\mathcal{D}_A^{-\theta}$. Observe that the function $L_{\alpha,j}\psi$ satisfies the equation
\begin{equation}\label{eq:Lj}
i\partial_t L_{\alpha,j}\psi=-\Delta_AL_{\alpha,j}\psi+L_{\alpha,j}F+i[\partial_t,L_{\alpha,j}]\psi+[L_{\alpha,j},-\Delta_A]\psi.
\end{equation}
Moreover, a direct computation yields
\begin{equation}\label{eq:comm_Lj}
[L_{\alpha,j},-\Delta_A]=\mathcal{D}_A^{-\theta}\Big(ih_{\alpha,j}\Theta_j+2\partial_{A_j}h_{\alpha,j}'\partial_{A_j}-h_{\alpha,j}''\partial_{A_j}\Big)\mathcal{D}_A^{-\theta},
\end{equation}
where we set 
\begin{equation}\label{def_theta}
\Theta_j:=\sum_{k=1}^3(\partial_jA_k-\partial_kA_j)\partial_{A_k}+\partial_{A_k}(\partial_jA_k-\partial_kA_j).
\end{equation}
Combining \eqref{eq:Lj} and \eqref{eq:comm_Lj}, and using the self-adjointness of $\mathcal{D}_A^{-\theta}$, we obtain 
\begin{equation*}
\begin{split}
i\partial_t(L_{\alpha,j}\psi,\psi)&=-\big(\Delta_AL_{\alpha,j}\psi-L_{\alpha,j}F-i[\partial_t,L_{\alpha,j}]\psi\\
&\quad\qquad-[L_{\alpha,j},-\Delta_A]\psi,\psi\big)+i(L_{\alpha,j}\psi,\partial_t \psi)\\
&=-(L_{\alpha,j}\psi,i\partial_t\psi+\Delta_A\psi)+(L_{\alpha,j}F,\psi)+i([\partial_t,L_{\alpha,j}]\psi,\psi)\\
&\quad\qquad+(i\mathcal{D}_A^{-\theta}h_{\alpha,j}\Theta_j\mathcal{D}_A^{-\theta}\psi,\psi)+2(\mathcal{D}_A^{-\theta}\partial_{A_j}h_{\alpha,j}'\partial_{A_j}\mathcal{D}_A^{-\theta}\psi,\psi)\\
&\quad\qquad-(\mathcal{D}_A^{-\theta}h_{\alpha,j}''\partial_{A_j}\mathcal{D}_A^{-\theta}\psi,\psi)\\
&=(L_{\alpha,j}\psi,F)+(L_{\alpha,j} F,\psi)+i([\partial_t,L_{\alpha,j}]\psi,\psi)+(h_{\alpha,j}\Theta_j\mathcal{D}_A^{-\theta}\psi,i\mathcal{D}_A^{-\theta}\psi)\\
&\quad\qquad -2(h_{\alpha,j}'\partial_{A_j}\mathcal{D}_A^{-\theta}\psi,\partial_{A_j}\mathcal{D}_A^{-\theta}\psi)-(\mathcal{D}_A^{-\theta}h_{\alpha,j}''\partial_{A_j}\mathcal{D}_A^{-\theta}\psi,\psi).
\end{split}
\end{equation*}
Integrating the above identity over $[0,T]$, observing that $h_{\alpha,j}'\equiv 1$ on $Q_{\alpha}$, and taking the sup over $\alpha\in\Z^3$ we deduce the bound
\begin{equation}\label{sanv}
\begin{split}
2\|\chi_{\alpha}\partial_{A_j}&\mathcal{D}_A^{-\theta}\psi\|_{\ell_{\alpha}^{\infty}L_T^2L^2(Q_{\alpha})}^2\leq\|(L_{\alpha,j}\psi,\psi)\|_{L_{\alpha,T}^{\infty}}\\
&+\|(L_{\alpha,j}\psi,F)+(L_{\alpha,j} F,\psi)\|_{\ell_{\alpha}^{\infty}L_T^1}\\
&+\|([\partial_t,L_{\alpha,j}]\psi,\psi)\|_{\ell_{\alpha}^{\infty}L_T^1}+\|(\Theta_j\mathcal{D}_A^{-\theta}\psi,\mathcal{D}_A^{-\theta}\psi)\|_{L_T^1}\\
&+\|(\mathcal{D}_A^{-\theta}h_{\alpha,j}''\partial_{A_j}\mathcal{D}_A^{-\theta}\psi,\psi)\|_{\ell_{\alpha}^{\infty}L_T^1}:=\mbox{I+II+III+IV+V}.
\end{split}
\end{equation}
Owing to the definition of $L_{\alpha,j}$, Cauchy Schwartz in the space variables and the equivalence of norms \eqref{equicono} we obtain
\begin{equation}\label{eq:125}
\mbox{I+II+V}\lesssim_T \langle\|A\|_{L_T^{\infty}H^1}\rangle^n\big(\|\psi\|_{L_T^{\infty}L^2}^2+\|F\|_{L_T^2H^{1-2\theta}}^2\big).
\end{equation}

Next, let us estimate term III. We start by computing
\begin{equation}\label{tuaq}
\begin{split}
[\partial_t,L_{\alpha,j}]&=[\partial_t,\mathcal{D}_A^{-\theta}]h_{\alpha,j}\partial_{A_j}\mathcal{D}_A^{-\theta}\\
&\quad+i\mathcal{D}_A^{-\theta}h_{\alpha,j}(\partial_tA_j)\mathcal{D}_A^{-\theta}+\mathcal{D}_A^{-\theta}h_{\alpha,j}\partial_{A_j}[\partial_t,\mathcal{D}_A^{-\theta}].
\end{split}
\end{equation}
For the first term in the right hand side of \eqref{tuaq} we have the estimate
\begin{equation}\label{comm_31}
\begin{split}
\|[\partial_t,\mathcal{D}_A^{-\theta}]h_{\alpha,j}\partial_{A_j}\mathcal{D}_A^{-\theta}\psi\|_{\ell_{\alpha}^{\infty}L_T^{\infty}L^2}&\lesssim \langle\|A\|_{\mathcal{X}_T}\rangle^n\|\partial_{A_j}\mathcal{D}_A^{-\theta}\psi\|_{L_T^{\infty}H^{\theta-1}}\\
&\lesssim \langle\|A\|_{\mathcal{X}_T}\rangle^n\|\psi\|_{L_T^{\infty}L^2},
\end{split}
\end{equation}
where we used Lemma \ref{lela} with $s_1=0$, $s_2=\theta-1$, $s_3=\theta$ and the equivalence of norms \eqref{equicono}. The third term in the r.h.s. of \eqref{tuaq} is treated similarly:
 \begin{equation}\label{comm_33}
\begin{split}
\|\mathcal{D}_A^{-\theta}h_{\alpha,j}\partial_{A_j}[\partial_t,\mathcal{D}_A^{-\theta}]\psi\|_{\ell_{\alpha}^{\infty}L_T^{\infty}L^2}&\lesssim \langle\|A\|_{\mathcal{X}_T}\rangle^n\|[\partial_t,\mathcal{D}_A^{-\theta}]\psi\|_{L_T^{\infty}H^{1-\theta}}\\
&\lesssim \langle\|A\|_{\mathcal{X}_T}\rangle^n\|\psi\|_{L_T^{\infty}L^2},
\end{split}
\end{equation}
where we used the bound \eqref{equicono} and Lemma \ref{lela} with $s_1=1-\theta$, $s_2=0$, $s_3=\theta$. For the second term in the r.h.s. of \eqref{tuaq} we have
\begin{equation}\label{comm_32}
\begin{split}
\|\mathcal{D}_A^{-\theta}h_{\alpha,j}&(\partial_tA_j)\mathcal{D}_A^{-\theta}\psi\|_{\ell_{\alpha}^{\infty}L_T^{\infty}L^2}\lesssim \langle\|A\|_{L_T^{\infty}H^1}\rangle^m\|(\partial_tA_j)\mathcal{D}_A^{-\theta}\psi\|_{L_T^{\infty}L^{4/3}}\\
&\lesssim \langle\|A\|_{L_T^{\infty}H^1}\rangle^m\|\partial_t A_j\|_{L_T^{\infty}L^2}\|\mathcal{D}_A^{-\theta}\psi\|_{L_T^{\infty}H^{\theta}}\lesssim \langle\|A\|_{\mathcal{X}_T}\rangle^n\|\psi\|_{L_T^{\infty}L^2},
\end{split}
\end{equation}
as it follows from \eqref{equicono} and the Sobolev embedding $H^{\theta}\hookrightarrow L^4$, valid as $\theta>\frac34$. Combining identity \eqref{tuaq} with the bounds \eqref{comm_31}, \eqref{comm_33} and \eqref{comm_32} we deduce
\begin{equation}\label{IIIc}
\|[\partial_t,L_{\alpha,j}]\psi\|_{\ell_{\alpha}^{\infty}L_T^{\infty}L^2}\lesssim \langle\|A\|_{\mathcal{X}_T}\rangle^n \|\psi\|_{L_T^{\infty}L^2}.
\end{equation}
Using \eqref{IIIc} and Cauchy-Schwartz in the space variables we eventually obtain
\begin{equation}\label{eq:III}
III\lesssim_T \langle\|A\|_{\mathcal{X}_T}\rangle^n\|\psi\|_{L_T^{\infty}L^2}^2.
\end{equation}
We are left to estimate term IV. To this aim, let us set $p:=p(\sigma)=6/(5-2\sigma)$, $q:=q(\sigma)=3/(\sigma-1)$, and observe that
\begin{equation}\label{IV_prel}
\begin{split}
IV&\leq 2\|\chi_{\alpha}\mathcal{D}_A^{-\theta}\psi\sum_{k=1}^3(\partial_jA_k-\partial_kA_j)\partial_{A_k}\mathcal{D}_A^{-\theta}\psi\|_{\ell_{\alpha}^1L_T^1L^1}\\
&\leq \frac12\|\chi_{\alpha}\nabla_{A}\mathcal{D}_A^{-\theta}\psi\|_{\ell_{\alpha}^{\infty}L_T^2L_x^2}^2+2\|\chi_{\alpha}\operatorname{rot}A\|_{\ell_{\alpha}^2L_T^{\infty}L^p}\|\chi_{\alpha}\mathcal{D}_A^{-\theta}\psi\|_{\ell_{\alpha}^2L_T^2L^{q}}^2\\
&\leq \frac12\|\chi_{\alpha}\nabla_{A}\mathcal{D}_A^{-\theta}\psi\|_{\ell_{\alpha}^{\infty}L_T^2L_x^2}^2+2\|A\|_{\mathcal{X}_T}\|\chi_{\alpha}\mathcal{D}_A^{-\theta}\psi\|_{\ell_{\alpha}^2L_T^2W^{\frac32-\sigma,6}}^2,
\end{split}
\end{equation}
where we used the definition \eqref{def_theta} of $\Theta_j$, an integration by parts and the bound $\|f\|_{L^1}\leq\|\chi_{\alpha}f\|_{\ell_{\alpha}^1L^1}$ in the first step, H\"older inequality in $\alpha,t,x$ together with Young inequality in the second step, and Sobolev embedding in the last step. Owing to the assumptions on $\delta$ and $\sigma$, we have $\frac32-\sigma\in\left(\frac{\theta-1}{2},\frac{2\theta-1}{2}\right)$. Hence, applying the bound \eqref{eq:improved-kt} with $s=0$ and $m=\frac32-\sigma$ we get
\begin{equation}\label{IV_inter}
\|\chi_{\alpha}\mathcal{D}_A^{-\theta}\psi\|_{\ell_{\alpha}^2L_T^2W^{3/2-\sigma,6}}\lesssim \langle\|A\|_{\mathcal{X}_T}\rangle^n\big(\|\psi\|_{L_T^{\infty}L^2}+\|F\|_{L_T^2H^{2\delta-2\sigma+1}}\big),
\end{equation}
which combined with \eqref{IV_prel} yields
\begin{equation}\label{IV_def}
IV-\frac12\|\chi_{\alpha}\nabla_{A}\mathcal{D}_A^{-\theta}\psi\|_{\ell_{\alpha}^{\infty}L_T^2L_x^2}^2\lesssim_T \langle\|A\|_{\mathcal{X}_T}\rangle^n\big(\|\psi\|_{L_T^{\infty}L^2}+\|F\|_{L_T^2H^{2\delta-2\sigma+1}}\big).
\end{equation}
Combining estimates \eqref{sanv}, \eqref{eq:125}, \eqref{eq:III} and \eqref{IV_def}, and summing over $j$, we eventually obtain
\begin{equation}\label{eq:smooth_mass}
\|\chi_{\alpha}\nabla_A\mathcal{D}_A^{-\theta}\psi\|_{\ell_{\alpha}^{\infty}L_T^2L^2_x}\lesssim_T\langle\|A\|_{\mathcal{X}_T}\rangle^n\Big(\|\psi\|_{L_T^{\infty}L^2}+\|F\|_{L_T^2H^{-1+2\delta}}\Big).
\end{equation}

Let us consider now the generic case $s\in[0,2)$. Observe that $\mathcal{D}^s_A\psi$ satisfies the equation
$$i\partial_t(\mathcal{D}^s_A\psi)=-\Delta_A(\mathcal{D}^s_A\psi)+i[\partial_t,\mathcal{D}_A^s]\psi+\mathcal{D}^s_A F.$$
Using estimate \eqref{eq:smooth_mass} we obtain
\begin{equation}\label{eq:smooth_intermass}
\begin{split}
\|\chi_{\alpha}\nabla_A\mathcal{D}_A^{s-\theta}&\psi\|_{\ell_{\alpha}^{\infty}L_T^2L^2_x}\lesssim_T\langle\|A\|_{\mathcal{X}_T}\rangle^n\Big(\|\mathcal{D}^s_A\psi\|_{L_T^{\infty}L^2}+\\
&\|[\partial_t,\mathcal{D}_A^s]\psi\|_{L_T^2H^{-1+2\delta}}+\|\mathcal{D}^s_A F\|_{L_T^2H^{-1+2\delta}}\Big).
\end{split}
\end{equation}
For the commutator term we have the estimate
\begin{equation}\label{final_commu}
\begin{split}
\|[\partial_t,\mathcal{D}_A^s]\psi\|_{L_T^2H^{-1+2\delta}}&=\|\mathcal{D}_A^s[\partial_t,\mathcal{D}_A^{-s}]\mathcal{D}_A^s\psi\|_{L_T^2H^{-1+2\delta}}\\
&\lesssim_T \langle\|A\|_{\mathcal{X}_T}\rangle^n\|[\partial_t,\mathcal{D}_A^{-s}]\mathcal{D}_A^s\psi\|_{L_T^{\infty}H^{s-1+2\delta}}\\
&\lesssim \langle\|A\|_{\mathcal{X}_T}\rangle^n\|\mathcal{D}_A^s\psi\|_{L^2}\lesssim \langle\|A\|_{\mathcal{X}_T}\rangle^n\|\psi\|_{L_T^{\infty}H^s},
\end{split}
\end{equation}
where we used the identity $[X,Y]=X[Y,X^{-1}]X$ in the first step, the equivalence of norms \eqref{equicono} in the second and last steps, and Lemma \ref{lela} with $s_1=s-1+2\delta$, $s_2=0$ and $s_3=s$ in the third step. Combining \eqref{eq:smooth_intermass} and \eqref{final_commu}, and using the equivalence of norms \eqref{equicono} we deduce
\begin{equation}\label{eq:smooth_mass_s}
\|\chi_{\alpha}\nabla_A\mathcal{D}_A^{s-\theta}\psi\|_{\ell_{\alpha}^{\infty}L_T^2L^2_x}\lesssim_T\langle\|A\|_{\mathcal{X}_T}\rangle^n\Big(\|\psi\|_{L_T^{\infty}H^s}+\|F\|_{L_T^2H^{s-1+2\delta}}\Big).
\end{equation}
Next we observe that, in view of the bounds \eqref{equicono} and \eqref{eq:disin} we have
\begin{equation}\label{comparison}
\begin{split}
\|\chi_{\alpha}\psi\|_{\ell_{\alpha}^{\infty}L_T^2H^{s+\delta}}&\lesssim \|\chi_{\alpha}\nabla\psi\|_{\ell_{\alpha}^{\infty}L_T^2H^{s-\theta}}\\
&\lesssim \langle\|A\|_{L_T^{\infty}H^1}\rangle^n \|\chi_{\alpha}\nabla_A\mathcal{D}_A^{s-\theta}\psi\|_{\ell_{\alpha}^{\infty}L_T^2L^2_x} + \|A\psi\|_{L_T^2H^{s-\theta}}.
\end{split}
\end{equation}
Moreover, owing to the assumption $\delta<2-s$, we also have
\begin{equation}\label{eq:remainder}
\|A\psi\|_{L_T^2H^{s-\theta}}\lesssim_T \langle\|A\|_{L_T^{\infty}H^1}\rangle\|\psi\|_{L_T^{\infty}H^s},
\end{equation}
as it follows from the fractional Leibniz rule and Sobolev embedding. Combining \eqref{comparison} together with \eqref{eq:smooth_mass_s} and \eqref{eq:remainder} we eventually deduce the smoothing estimate \eqref{eq:local_smoothing} for $s\in(0,2)$. The proof is complete.
\end{proof}

\subsection{Local smoothing for the Maxwell-Schr\"odinger system}
We apply now the local smoothing estimate \eqref{eq:local_smoothing} to the case of the non-linear Maxwell-Schr\"odinger system. Preliminary, we need a refined version of the Strichartz estimate \ref{eq:kg_stri} for the wave equation, analogous to \cite[Lemma 2.2]{Wada2012}, which exploits the finite speed of propagation for hyperbolic equations.

\begin{lemma}\label{le:uno_smoothing}
	Let $T>0$, $\sigma\geq 1$, and let $(q_0,r_0)$ be a wave-admissible pair. Fix moreover $(A_0,A_1)\in\Sigma^{\sigma}$ and $F\in L_{T}^{q'_0}W^{\sigma-1+2/q_0,r'_0}$, and let $A$ be the solution to $\square A=F$ with initial data $A(0)=A_0$, $\partial_tA(0)=A_1$. Then for every wave-admissible pair $(q,r)$, we have the estimate
	\begin{equation}\label{eq:kg_stri_smooth}
		\max_{k=0,1}\|\chi_{\alpha}\partial_t^k A\|_{\ell_{\alpha}^2L_T^{q}W^{\sigma-k-2/q,r}}\lesssim\|(A_0,A_1)\|_{\Sigma^{\sigma}}+\|F\|_{L_T^{q'_0}W^{\sigma+2/q_0-1,r'_0}}.
	\end{equation}
\end{lemma}

\begin{proof}
	We recall that, by finite speed of propagation (here the speed of light is normalized to $c=1$), the quantity $A(t,x)$ is determined by the values of $A_0,A_1$ and $F(t)$ in the ball $B_x(t)$ of center $x$ and radius $t$. For every fixed $\alpha\in\Z^3$, let $\widetilde{\chi}_{\alpha,T}:\R^3\to [0,1]$ be a smooth function such that
	$$
	\widetilde{\chi}_{\alpha,T}(x)=\begin{cases}
		1&\mbox{if }\operatorname{dist}(x,\operatorname{supp}\chi_{\alpha})\leq T\\
		0&\mbox{if }\operatorname{dist}(x,\operatorname{supp}\chi_{\alpha})\geq T+1,
	\end{cases}$$
	and consider the Cauchy problem
	\begin{equation}\label{eq:acut}
		\square\widetilde{A}_{\alpha,T}=\widetilde{\chi}_{\alpha,T}F,\quad (\widetilde{A}_{\alpha,T},\partial_t \widetilde{A}_{\alpha,T})(0)=\widetilde{\chi}_{\alpha,T}(A_0,A_1).
	\end{equation}
	In view of the observation above and the definition of the cut-off $\widetilde{\chi}_{\alpha,T}$, we have that $A(t,x)=\widetilde{A}_{\alpha,T}(t,x)$ for every $(t,x)\in (0,T)\times\operatorname{supp}\chi_{\alpha}$. In particular, applying the standard Strichartz estimates \eqref{eq:kg_stri} to the Cauchy problem \eqref{eq:acut} we deduce
	\begin{equation}\label{eq:blf}
		\max_{k=0,1}\|\chi_{\alpha}\partial_t^k A\|_{L_T^{q}W^{\sigma-k-2/q,r}}^2\lesssim \|\widetilde{\chi}_{\alpha,T}(A_0,A_1)\|_{\Sigma^\sigma}+\|\widetilde{\chi}_{\alpha,T}F\|_{L_T^{q'_0}W^{\sigma+2/q_0-1,r'_0}}.
	\end{equation}
	Observe moreover that, for any given $\alpha\in\Z^3$, the number of $\widetilde{\alpha}\in\Z^3$ such that $\operatorname{supp}\chi_{\widetilde{\alpha}}\cap \operatorname{supp}\widetilde{\chi}_{\alpha,T}\neq\emptyset$ is bounded by $C\langle T\rangle^3$, for some constant $C$ uniform in $\alpha$. Hence, summing the bound \eqref{eq:blf} over $\alpha\in\Z^3$ we obtain
	\begin{equation*}
		\begin{split}
			\max_{k=0,1}\|\chi_{\alpha}\partial_t^k A\|_{\ell_{\alpha}^2L_T^{q}W^{\sigma-k-2/q,r}}&\lesssim\|\widetilde{\chi}_{\alpha,T}(A_0,A_1)\|_{\ell_{\alpha}^2\Sigma^\sigma}+\|\widetilde{\chi}_{\alpha,T}F\|_{\ell_{\alpha}^2L_T^{q'_0}W^{\sigma+2/q_0-1,r'_0}}\\
			&\lesssim_T\|\chi_{\alpha}(A_0,A_1)\|_{\ell_{\alpha}^2\Sigma^\sigma}+\|\chi_{\alpha}F\|_{\ell_{\alpha}^2L_T^{q'_0}W^{\sigma+2/q_0-1,r'_0}}\\
			&\lesssim \|(A_0,A_1)\|_{\Sigma^\sigma}+\|F\|_{L_T^{q'_0}W^{\sigma+2/q_0-1,r'_0}},
		\end{split}
	\end{equation*}
	where in the last step we used Minkowski inequality to interchange the order of variables in the $\ell_{\alpha}^2L_T^{q'_0}$-norm, the embedding $\ell_{\alpha}^{r_0'}\hookrightarrow \ell_{\alpha}^2$, and the bound \eqref{eq:disin}. The proof is complete.
\end{proof}

We are now ready to prove the main result of this section.

\begin{proof}[Proof of Proposition \ref{pr:smoot_MS}]
Applying the refined Strichartz estimates \eqref{eq:kg_stri_smooth} to the equation $\square A=\mathbb{P}J$, we obtain the bound 
	\begin{equation}\label{eq:3.9AMS}
		\|\chi_{\alpha}A\|_{\ell_{\alpha}^2L_T^{q}W^{\sigma-2/q,r}}\lesssim\|(A_0,A_1)\|_{\Sigma^{\sigma}}+\|\mathbb{P}J\|_{L_T^{6/5}W^{\sigma-2/3,3/2}},
	\end{equation}
valid for every wave-admissible pair $(q,r)$. Moreover, using Lemma \ref{le:tre}, the fractional Leibniz rule \eqref{frac_leibniz}, and observing that $\sigma-2/3\leq 1/2$, we get
	\begin{equation}\label{dabc}
		\begin{split}
			\|\mathbb{P}J\|_{L_T^{6/5}W^{\sigma-2/3,3/2}}&\lesssim_{\langle T\rangle^n}\|\mathbb{P}(\bar{\psi}\nabla \psi)\|_{L_T^2W^{\sigma-2/3,3/2}} + \|A|\psi|^2\|_{L_T^{\infty}W^{\sigma-2/3,3/2}}\\
			&\lesssim \|\psi\|_{L_T^2W^{\sigma-2/3,6}}\|\nabla \psi\|_{L_T^{\infty}L^2}+\|A\|_{L_T^{\infty}W^{1/2,3}}\|\psi^2\|_{L_T^{\infty}L^3}\\
			&\quad+\|A\|_{L_T^{\infty}L^6}\|\psi\|_{L_T^{\infty}W^{1/2,3}}\|\psi\|_{L_T^{\infty}L^6}\\
			&\lesssim (\|\psi\|_{L_T^2W^{1/2,6}}+\|A\|_{L_T^{\infty}H^1})\langle\|\psi\|_{L_T^{\infty}H^1}^2\rangle.
		\end{split}
	\end{equation}
Combining \eqref{eq:3.9AMS}, \eqref{dabc} and the a priori estimate \eqref{eq:apriori_uno} for $\|\psi\|_{L_T^2W^{1/2,6}}$ we deduce that for every wave-adissible pair $(q,r)$ we have the bound
	\begin{equation}\label{uhgt}
		\|\chi_{\alpha}A\|_{\ell_{\alpha}^2L_T^{q}W^{\sigma-2/q,r}} \lesssim_T \langle\|(\psi,A)\|_{L_T^{\infty}(H^1\times H^1)}\rangle^n\langle\|(A_0,A_1)\|_{\Sigma^{\sigma}}\rangle^n.
	\end{equation}
In particular, in view of \eqref{uhgt} and the definition \eqref{def:xt} of the $\mathcal{X}_T$-norm, we get
	$$\|A\|_{\mathcal{X}_T}\lesssim_T \langle\|(\psi,A)\|_{L_T^{\infty}(H^1\times H^1)}\rangle^n\langle\|(A_0,A_1)\|_{\Sigma^{\sigma}}\rangle^n.$$
Using the inequality above, Proposition \ref{pr:local_smoothing} and the smoothing estimate \eqref{sm-du-st} for the inhomogeneous Schr\"odinger equation, we obtain the for every $\delta\in(0,\sigma-1)$
	\begin{equation}\label{eteuno}
		\begin{split}
			\|\chi_{\alpha}\psi\|_{\ell_{\alpha}^{\infty}L_T^{2}H^{1+\delta}} & \lesssim_T \langle\|(\psi,A)\|_{L_T^{\infty}(H^1\times H^1)}\rangle^n\langle\|(A_0,A_1)\|_{\Sigma^{\sigma}}\rangle^n\\
			&\quad\times\|\phi\psi\|_{L_T^2H^{2\delta}}+\||\psi|^{2(\gamma-1)}\psi\|_{L_T^2W^{1/2+\delta,6/5}}.
		\end{split}
	\end{equation}
 Since $2\delta<2(\sigma-1)<\frac13$, we deduce from \eqref{esti_conv} that
\begin{equation}\label{etedue}
\|\phi\psi\|_{L_T^2H^{2\delta}}\lesssim_T \|\psi\|_{L_T^{\infty}H^1}^3.
\end{equation}
 For the pure-power term, we first observe that we have the embedding 
$$H^1\cap W^{1/2,6}\hookrightarrow W^{1/2+(3-\gamma)/2,6/(7-2\gamma)\vee 2}.$$
Since $\delta<\frac{3-\gamma}{2}$, Lemma \ref{esti_pure}, the embedding above and the estimate \eqref{eq:apriori_uno} yield
\begin{equation}\label{etetre}
	\begin{split}
\||\psi|^{2(\gamma-1)}\psi\|_{L_T^2W^{1/2+\delta,6/5}}&\lesssim_T \|\psi\|^{2(\gamma-1)}_{L_T^{\infty}L^{6(\gamma-1)\wedge 6}}\|\psi\|_{L_T^2W^{1/2+(3-\gamma)/2, 6/(7-2\gamma)\vee 2}}\\
 & \lesssim_T \langle\|(\psi,A)\|_{L_T^{\infty}(H^1\times H^1)}\rangle^n\langle\|(A_0,A_1)\|_{\Sigma^{\sigma}}\rangle^n.
 \end{split}
\end{equation}
 
 Combining \eqref{eteuno},\eqref{etedue} and \eqref{etetre} we eventually deduce the local smoothing estimate \eqref{eq:losmoMS}.
\end{proof}

\section{Proof of the main results}\label{sec:fin}
This section is devoted to the proof of our main results, namely Theorem \ref{th:main} and Theorem \ref{th:stability}. Let us start with the existence of global, finite energy, weak solutions to the QMHD system.
\begin{proof}[Proof of Theorem \ref{th:main}]
In view of Proposition \ref{pr:exMS}, there exists a global, weak $M^{1,1}$-solution $(\psi,A)$ to the Maxwell-Schr\"odinger system \eqref{eq:MS}, with initial data $(\psi_0,A_0,A_1)$, satisfying the uniform energy bound $\sup_{t\in\R^+}\mathcal{E}(t)\leq\mathcal{E}(0)$, and such that for every $T>0$
\begin{equation}\label{eq:unip}
	\|(\psi,A,\partial_t A)\|_{L_T^{\infty}M^{1,1}}\lesssim_T \|(\psi_0,A_0,A_1)\|_{M^{1,1}}.
\end{equation}
Associated to $(\psi,A)$, we consider the hydrodynamic variables $(\rho,J,E,B)$ defined by \eqref{eq:mad} and \eqref{eq:em_f}. We are going to show that $(\rho,J,E,B)$ is a global, finite energy, weak solution to the Cauchy problem \eqref{eq:qmhd}-\eqref{eq:qmhd_id}, in the sense of Definition \ref{def:fews}.

\begin{itemize}
\item[(i)] Using the relation $\nabla\sqrt{\rho}=\RE(\bar{\varphi}\nabla\psi)$, with $\varphi\in P(\psi)$, we obtain
\begin{equation}\label{eq:ii}
	\|\sqrt{\rho}\|_{L_T^2H^1}\lesssim_T\|\psi\|_{L_T^{\infty}H^1}\lesssim_T 1.
\end{equation}
As discussed in Subsection 2.1, the relation $J=\sqrt{\rho}\Lambda$ is a direct consequence of the definition $\Lambda:=\IM(\bar{\varphi}\nabla_A\psi)$ of the hydrodynamic state $\Lambda$. Using in addition the equivalence of norms \eqref{equicono} we get
\begin{equation}\label{eq:Aii}
	\|\Lambda\|_{L^{2}_TL^2}\lesssim_T\|\psi\|_{L_T^{\infty}H_A^1}\lesssim \langle\|A\|_{L_T^{\infty}H^1}\rangle\|\psi\|_{L_T^{\infty}H^1}\lesssim_T 1.
\end{equation}
\item[(ii)] Using the expression $(E,B)=(-\nabla\phi-\partial_t A,\nabla\wedge A)$ for the electromagnetic field and the bound $\|\phi(\psi)\|_{H^1}\lesssim \|\psi\|^2_{H^1}$ we deduce
\begin{equation}\label{iEB}
	\begin{split}
\|(E,B)\|_{L_T^2(L^2\times L^2)}&\lesssim_T \|\nabla\phi+\partial_t A\|_{L_T^{\infty}L^2}+\|\nabla\wedge A\|_{L_T^{\infty}L^2}\\
&\lesssim \|\psi\|_{L_T^{\infty}H^1}^2+\|A\|_{L_T^{\infty}\Sigma^1}\lesssim_T 1.
\end{split}
\end{equation}
\item[(iii)] The condition $F_L\in L^1_{\mathrm{loc}}(\R_t^+\times\R_x^3)$ is guaranteed by Proposition \ref{pr:wdlf}.
\item[(iv)] The weak formulation \eqref{eq:continuity} of the continuity equation has been proved in Proposition \ref{pr:continuity}.
\item[(v)] The weak formulation \eqref{eq:momentum} of the momentum equation has been proved in Proposition \ref{pr:momentum}.
\item[(vi)] The distributional validity of the Maxwell equations \eqref{eq:Max} follows by a direct computation, owing to the distributional identity $\square A=\mathbb{P}J$.
\item[(vii)] We start by observing that
$$\nabla\wedge J=\IM(\nabla\bar{\psi}\wedge\nabla\psi)-\nabla|\psi|^2\wedge A-B|\psi|^2,$$
which yields
\begin{equation}\label{gic1}
	\nabla\wedge J+\rho B=\IM(\nabla\bar{\psi}\wedge\nabla\psi)-2\sqrt{\rho}\,\nabla\sqrt{\rho}\wedge A.
	\end{equation}
On the other hand, we have
\begin{equation}\label{gic2}
2\nabla\sqrt{\rho}\wedge\Lambda=2\nabla\sqrt{\rho}\wedge\IM(\bar{\varphi}\nabla\psi)-2\nabla\sqrt{\rho}\wedge(A\sqrt{\rho}).
\end{equation}
Moreover, we have the chain of identities
\begin{equation}\label{gic3}
	\begin{split}
\IM(\nabla\bar{\psi}\wedge\nabla\psi)&=2\IM(\bar{\varphi}\nabla\sqrt{\rho}\wedge\sqrt{\rho}\,\nabla\phi)=2\nabla\sqrt{\rho}\wedge\IM(\bar{\varphi}\nabla\psi).
	\end{split}
\end{equation}
Combining \eqref{gic1}, \eqref{gic2} and \eqref{gic3} we obtain the generalized irrotationality condition \eqref{eq:gen_irr}.
\end{itemize}
The proof is complete.
\end{proof} 

Next we consider the stability result. Here the gain of regularity provided by the local-smooting estimate \eqref{eq:losmoMS} will be crucial. Though it is not difficult to prove the weak $H^1$-continuity of the map $\psi\mapsto\sqrt{\rho}$, a similar result for $\Lambda$ in fact fails to be true. Fix indeed an open, finite ball $B\subseteq\R^3$, and consider e.g.~the sequences $A_n=0$, $\psi_n=n^{-1}e^{inx_1}\in H^1(B)$. It is clear that $\psi_n\to 0$ in $L^2(B)$, and moreover $\nabla\psi_n=ie^{inx_1}\mathbf{e}_1$ weak converges to $0$ in $L^2(B)$. However, for every $n\in\N$, 
$$\Lambda_n=\IM(\bar{\phi}_n\nabla_{A_n}\psi_n)=\IM(e^{-inx_1}\cdot ie^{inx_1}\mathbf{e}_1)=\mathbf{e}_1,$$
and then $\Lambda_n\to \mathbf{e}_1\neq 0$ in $L^2(B)$. Estimate \eqref{eq:losmoMS} will provide enough compactness in order to recover the weak stability of $\Lambda$, and to prove the strong stability of the Lorentz force as well.

We will also need a classical compactness result, the Aubin-Lions Lemma (see e.g.~\cite[Section 7.3]{Roubicek_NPDEbook}).
\begin{lemma}\label{le:aubin}
	Let $\mathcal{X},\mathcal{Y}$ and $\mathcal{Z}$ be Banach spaces such that $\mathcal{X}\hookrightarrow \mathcal{Y}$ compactly and $\mathcal{Y}\hookrightarrow \mathcal{Z}$. Let us fix moreover $T>0$, $p,q\in[1,\infty)$, and let $U\subseteq L_T^p\mathcal{X}\cap W_T^{1,q}\mathcal{Z}$ be a bounded set in $L_T^p\mathcal{X}$. Then $U$ is relatively compact in $L_T^p\mathcal{Y}$.
\end{lemma}

We are ready to prove the stability result.

\begin{proof}[Proof of Theorem \ref{th:stability}]
Observe preliminary that $\big(\psi_0^{(n)},A_0^{(n)},A_1^{(n)}\big)$ weak converges, up to a subsequence, to some initial data $(\psi_0,A_0,A_1)$ in $\Sigma^{1,\sigma}$. Let $\big(\psi^{(n)},A^{(n)}\big)$ be the global, weak $M^{1,1}$-solution to the non-linear Maxwell-Schr\"odinger \eqref{eq:MS} with initial data $\big(\psi_0^{(n)},A_0^{(n)},A_1^{(n)}\big)$, associated to  $(\rho^{(n)},J^{(n)},E^{(n)},B^{(n)})$ through formulas \eqref{eq:mad} and \eqref{eq:em_f}. In view of estimate \eqref{eq:unip}, for any $T>0$ we have the uniform bound
\begin{equation}\label{equb}
\big\|\big(\psi^{(n)},A^{(n)},\partial_t A^{(n)}\big)\big\|_{L_T^{\infty}M^{1,1}}\lesssim_T \big\|\big(\psi_0^{(n)},A_0^{(n)},A_1^{(n)}\big)\big\|_{M^{1,1}}\lesssim_T 1.
\end{equation}
As a consequence (see e.g.~the proof of \cite[Proposition 4.3]{ADM}), one can deduce the existence of a global, weak $M^{1,1}$-solution $(\psi,A)$ to the system \eqref{eq:MS} with initial data $(\psi_0,A_0,A_1)$, such that for every $T>0$ 
$$\big(\psi^{(n)},A^{(n)},\partial_t A^{(n)}\big)\to(\psi,A,\partial_tA)\quad\mbox{weak* in }L_T^{\infty}M^{1,1}.$$
Using the uniform bound \eqref{equb}, estimate \eqref{eq:apriori_uno}, and interpolating the smoothing estimate \eqref{eq:losmoMS} with the $L_T^{\infty}H^1$-bound, we deduce the following properties:
\begin{enumerate}
\item the sequences $\|(A^{(n)},\partial_tA^{(n)})\|_{L_T^{\infty}\Sigma^{\widetilde{\sigma}}}$, $\|\psi^{(n)}\|_{L_T^2W^{1/2,6}}$ are uniformly bounded, for some $\widetilde{\sigma}\in(1,\frac76)$;
	\item there exits $\delta>0$ such that the sequence $\|\chi_{\alpha}\psi^{(n)}\|_{\ell_{\alpha}^{\infty}L_T^4H^{1+\delta}}$ is uniformly bounded.
\end{enumerate}
Finally, let $(\rho,J,E,B)$ be the hydrodynamic variables associated to $(\psi,A)$ through formulas \eqref{eq:mad} and \eqref{eq:em_f}. In view of the proof of Theorem \ref{th:main}, the quadruple $(\rho,J,E,B)$ is a global, bounded energy solution to the QMHD system \eqref{eq:qmhd}. We are now able to prove the  stability of both the hydrodynamic variables and the Lorentz force.

\textbf{(i)~stability of the hydrodynamic variables}. Let us fix $T>0$ and an arbitrary open, bounded set $\Omega\subseteq\R^3$. Using the local smoothing property (2), the compact embedding $H^{1+\delta}(\Omega)\hookrightarrow H^1(\Omega)$, and the fact that $\partial_t\psi^{(n)}\in L_T^{\infty}H^{-1}$, the Aubin-Lions Lemma \ref{le:aubin} implies 
\begin{equation}\label{eq:ssal}
	\psi^{(n)}\to\psi\mbox{ in }L_T^4H^1(\Omega).
\end{equation}
The strong continuity result \eqref{eq:ssp} for the Madelung transform then guarantees that
\begin{equation}\label{key1}
	\sqrt{\rho^{(n)}} \to \sqrt{\rho}\quad\mbox{in }L_T^4H^1(\Omega),\qquad \Lambda^{(n)} \to \Lambda\quad\mbox{in }L_T^4L^2(\Omega).
\end{equation}
Moreover, since the sequence $\|\psi^{(n)}\|_{L_T^{\infty}H^1}$ is bounded, then also $\|\sqrt{\rho^{(n)}}\|_{L_T^{\infty}H^1}$ and $\|\Lambda^{(n)}\|_{L_T^{\infty}L^2}$ are bounded, in view of the relation
$$\|\nabla\psi^{(n)}\|_{L^2}=\|\nabla\sqrt{\rho^{(n)}}\|_{L^2}+\|\Lambda^{(n)}\|_{L^2}.$$
As a consequence, the limits in \eqref{key1} actually hold as weak* limits in $L_T^{\infty}H^1$ and $L_T^{\infty}L^2$, respectively.

Finally, the weak* convergence $(E^{(n)},B^{(n)})\rightharpoonup(E,B)$ in $L_T^{\infty}(L^2\times L^2)$ easily follows from formula \eqref{eq:em_f} for the electromagnetic field and the weak* convergence $|\psi^{(n)}|^2\rightharpoonup |\psi|^2$ in $L_T^{\infty}L^{3/2}\hookrightarrow L_T^{\infty}H^{-1}$, which in turn follows by the boundedness in $L_T^{\infty}L^{3/2}$ of $|\psi^{(n)}|^2$ and the strong convergence \eqref{eq:ssal}.

\textbf{(ii)~stability of the Lorentz force}. Let us fix $T>0$ and an arbitrary open, bounded set $\Omega\subseteq\R^3$. Using the Aubin-Lions Lemma \ref{le:aubin}, property (1), the compact embedding $\Sigma^{\widetilde{\sigma}}(\Omega)\hookrightarrow \Sigma^1(\Omega)$, and the fact that $(\partial_tA^{(n)},\partial_{tt}A^{(n)})\in L_T^{\infty}(L^2\times H^{-1})$, we obtain the strong convergence
\begin{equation}\label{eq:saab}(A^{(n)},\partial_t A^{(n)})\to (A,\partial_t A)\mbox{ in }L_T^4\Sigma^1(\Omega).
\end{equation}
Moreover, we deduce from \eqref{eq:ssal} that $\rho^{(n)}\to\rho$ and  $\nabla\phi^{(n)}\to\nabla\phi$ in $L_T^2L^2(\Omega)$. Combining everything, we get 
\begin{equation*}\label{conv_rho}
\rho^{(n)}E^{(n)}=-\rho^{(n)}(\nabla\phi^{(n)}+\partial_tA^{(n)})\to -\rho(\nabla\phi+\partial_tA)=\rho E\mbox{  in }L_T^1L^1(\Omega).
\end{equation*}

We are left to show the stability of the term $J\wedge B$. To this aim, let us fix $p\in (2,\frac{6}{5-2\widetilde{\sigma}})$, and let $q\in(2,\infty)$, $r\in(1,3)$ be such that $\frac1p+\frac1q=\frac12$ and $\frac1p+\frac1r=\frac56$. Using property (1), the Aubin-Lions Lemma \ref{le:aubin} yields $\psi^{(n)}\to\psi$ in $L_T^{2}L^{q}(\Omega)$ and $B^{(n)}\to B$ in $L_T^4L^{p}(\Omega)$, owing to the compact embedding $W^{1/2,6}(\Omega)\hookrightarrow L^{q}(\Omega)$ and $H^{\widetilde{\sigma}-1}\hookrightarrow L^p$ respectively. Moreover, \eqref{eq:saab} and \eqref{eq:ssal} yield
\begin{itemize}
	\item[-] $A^{(n)}\to A$ in $L^4L^6(\Omega)$, $\nabla\psi^{(n)}\to \nabla\psi$ in $L^4L^2(\Omega)$, $\rho^{(n)}\to\rho$ in $L_T^{2}L^{r}(\Omega)$.
	\end{itemize}
Combining everything we obtain
$$J^{(n)}=\IM\big(\overline{\psi^{(n)}}\nabla\psi^{(n)}\big)-\rho^{(n)}A^{(n)}\to \IM\big(\overline{\psi}\nabla\psi\big)-\rho A=J\mbox{ in } L_T^{4/3}L^{p'}(\Omega),$$
which together with the strong convergence $B^{(n)}\to B$ in $L_T^4L^{p}(\Omega)$ implies that $J^{(n)}\wedge B^{(n)}\to J\wedge B$ in $L_T^1L^1(\Omega)$. The proof is complete.
\end{proof}

\end{document}